\documentclass[12pt]{article}

\usepackage{amssymb,amsfonts,amsmath,amsthm}
\usepackage{graphicx}
\usepackage{color}
\usepackage{enumerate}
\usepackage{a4wide}
\usepackage{amssymb, amsfonts, amsthm,amsmath}
\usepackage{amsfonts,color,amssymb,mathrsfs,stmaryrd}
\usepackage{setspace}
\usepackage{graphicx}
\usepackage{tikz}
\usepackage{float}
\usepackage{subfig}

\usetikzlibrary{decorations.pathreplacing}

\newcommand{\vol}{{\rm vol}}
\newcommand{\q}{\mathrm{mod}}
\newcommand{\A}{\mathcal{A}}

\newcommand{\eps}{\varepsilon}

\newcommand{\Ex}[1]{\mathbb{E} \left[\, #1\,\right]}

\newcommand{\Var}{\mathrm{Var}}

\newcommand{\Bcal}{\mathcal{B}}

\newcommand{\Ball}[2]{B_{#2} (#1)}

\newcommand{\E}[1]{\mathbb{E}\left(#1 \right)}

\newcommand{\PP}[2]{\mathrm{P}_{#1,#2}}
\newcommand{\PPna}{\mathrm{P}_{\alpha,\nu,n}}
\newcommand{\PPnan}{\mathrm{P}_{\alpha,\nu,n}^{(>\delta R)}}

\newcommand{\cP}{\mathcal{P}}

\newcommand{\V}{V_n}
\newcommand{\D}{\mathcal{D}_R}

\newcommand{\G}{\mathcal{G}}
\newcommand{\Gnan}{\mathcal{G}(n;\alpha,\nu)}
\newcommand{\Pnan}{\mathcal{P}(n;\alpha,\nu)}
\newcommand{\PBnan}[1]{\mathcal{P}_{#1} (n;\alpha,\nu)}
\newcommand{\Bnan}[1]{\mathcal{B}_{#1}(n;\alpha,\nu)}

\newcommand{\Po}[1]{\mathrm{Po} \left(#1\right)}

\newtheorem{theorem}{Theorem}  

\newtheorem{lemma}[theorem]{Lemma}

\newtheorem{claim}[theorem]{Claim}

\numberwithin{theorem}{section}
\numberwithin{equation}{section}

\title{The modularity of random graphs on the hyperbolic plane
\footnote{ \small 2010 \emph{Mathematics Subject Classification}: Primary: 05C80 Secondary: 05C12, 05C82.
\small \emph{Keywords}: random geometric graphs, hyperbolic plane, complex networks, modularity.}
}

 \author{
 Jordan Chellig\footnote{School of Mathematics, University of Birmingham, United Kingdom, e-mail: \texttt{JAC555@bham.ac.uk}}
 \qquad
Nikolaos Fountoulakis\footnote{School of Mathematics, University of Birmingham, United Kingdom, e-mail: \texttt{n.fountoulakis@bham.ac.uk}}
\qquad
Fiona Skerman\footnote{Department of Mathematics, Uppsala University
e-mail: \texttt{fiona.skerman@math.uu.se}}
}

\date{\today}

\begin{document}

\maketitle

\begin{abstract} 
Modularity is a quantity which has been introduced in the context of complex networks in order to quantify how close a network is to an ideal modular network in which the nodes form small interconnected communities that are joined together with relatively few edges. In this paper, we consider this quantity on a recent probabilistic model of complex networks introduced by Krioukov et al. (\emph{Phys. Rev. E} 2010). 

This model views a complex network as an expression of hidden hierarchies, encapsulated by an underlying hyperbolic space. For certain parameters, this model was proved to have typical features that are observed in complex networks such as power law degree distribution, bounded average degree, clustering coefficient 
that is asymptotically bounded away from zero, and ultra-small typical distances. 
In the present work, we  investigate its modularity and we show that, in this regime, it converges to 1 in probability. 
\end{abstract}

\section{Introduction}

M. Granovetter, in his pioneering analysis of social networks~\cite{ar:Gran73}, pointed out that a fundamental feature of social networks is the distinction between weak links and strong links. These reflect the intensity of interaction between two individuals, which may be dependent on measures such as the frequency of interaction. 
Granovetter pointed out that an individual is more likely to interact with other individuals through the strong links. 
This is expressed in terms of structural features of the social network, whereby individuals belong to 
 communities, tightly knit by strong links, and these communities are typically joined through weak links. 

These ideas postulate that a fundamental characteristic of social networks is the existence of communities or modules within such a network. These are mutually disjoint subsets of nodes/individuals which have high density, but are joined to other modules by few edges. 

Identifying such a partition in a given social network or any other complex network is computationally challenging. 
But before we set out to find algorithms that give even an approximate solution to this problem, one needs to 
quantify what is a good partition of the node set of a given network. Such a quantification was given by 
Newman and Girvan~\cite{NewmanGirvan} and is called the \emph{modularity score} of a given partition. 
The highest modularity score among all partitions is called the \emph{modularity} of a network (cf. Section~\ref{sec:mod} for the precise definition). The most popular algorithms used to cluster large network data use the modularity score as a quality function for partitions (see for example~\cite{ar:lancichinetti}).

In this paper, we investigate the modularity of a recent model of complex networks in which a network is sampled as a geometric random graph on the hyperbolic plane.

\subsection{The KPKBV model: a geometric framework for complex networks}
Krioukov et al.~\cite{ar:Krioukov} introduced a model of random geometric graphs on the hyperbolic plane as a model of complex networks, which we abbreviate as the \emph{KPKBV model} after its inventors. 
This is based on the assumption that the geometry of the hyperbolic plane  can accommodate the hidden hierarchy of a complex network and its intrinsic inhomogeneity. 
 Their basic assumption is that the hierarchies
that are present in a complex network induce a tree-like structure, and this suggests that 
there is an underlying geometry of a complex network which is the hyperbolic.

There are several representations of the standard hyperbolic plane $H^2_{-1}$ of curvature $-1$. 
In this paper, we shall use the
Poincar\'e unit disc representation, which is simply the open disc of radius one, that is,
$\{(u,v) \in \mathbb{R}^2 \ : \ u^2+v^2 < 1 \}$,
which is equipped with the hyperbolic metric: ${4}~\frac{du^2 + dv^2}{ (1-u^2-v^2)^2}$.
This is a standard formulation of the hyperbolic plane.


In particular, a suitable integration of the metric shows that
the length of a circle of (hyperbolic) radius $r$ (centered at the origin)  is
$2\pi~\sinh (r)$, whereas the area of this circle (centered at the origin) is
$2\pi (\cosh ( r) - 1)$. Hence, a fundamental difference with the Euclidean plane is that
volumes grow exponentially.

The KPKBV model introduced by Krioukov et al.~\cite{ar:Krioukov} yields a random geometric graph on 
$H^2_{-1}$.
Consider the Poincar\'e disc representation of the hyperbolic plane $H^{2}_{-1}$.
The random graph will have $n$ vertices and this is the parameter we take asymptotics with respect to.

Let $\nu >0$ be a fixed constant and let $R=R(n)>0$ satisfy $n= \nu e^{R /2}$.
(It turns out that the parameter $\nu$ determines the average degree of the random graph.)
Consider the disc $\D$ of hyperbolic radius $R$ centered at the origin of the Poincar\'e disc (that is, the set of points of the Poincar\'e disc at hyperbolic distance at most $R$ from its origin).

We take a random set of points of size $n$ that are the outcomes of the $i.i.d.$ random variables $v_1,\ldots , v_n$ taking values on $\D$. 
(We will be referring to the random variables $v_i$ as vertices, meaning their values on $\D$.)
More specifically,  assume that $v_1$ has \emph{polar} coordinates $(r, \theta)$. The angle $\theta$ is uniformly distributed in $(0,2\pi]$ and the probability density function of
$r$, which we denote by $\rho_n (r)$, is determined by a parameter $\alpha >0$ and is equal to
\begin{equation} \label{eq:pdf}
 \rho(r) = \rho_n (r) = \begin{cases}
\alpha \frac{\sinh  (\alpha r )}{ \cosh (\alpha R ) - 1}, & \mbox{if $0\leq r \leq R$} \\
0, & \mbox{otherwise}
\end{cases}.
\end{equation}
The aforementioned formulae for the area and the length of a circle of a given radius imply that if we set
$\alpha =1$, the distribution described in~\eqref{eq:pdf} is the uniform distribution on $\D$ (under the hyperbolic metric).
For general $\alpha > 0$ Krioukov et al.~\cite{ar:Krioukov} called this the \emph{quasi-uniform} distribution on $\D$.
Let us remark that in fact  this is the uniform distribution on a disc of hyperbolic radius $R$ within $H^2_{-\alpha^2}$ (the hyperbolic plane that has curvature $-\alpha^2$).


Given the point process $\V = \{v_1,\ldots, v_n\}$ on $\D \subset H_{-1}^2$ and the fixed parameters $\alpha$ and $\nu$ we define the random graph $\G (n; \alpha, \nu)$  on the point-set of $\V$, where two distinct points
form an edge if and only if they are within (hyperbolic) distance $R$ from each other. Figure~\ref{fig:tube} shows the ball of radius $R$ around a point $p \in \D$, denoted by $B(p;R)$. 
Thus, any point/vertex of $\G (n;\alpha,\nu)$ that falls inside the shaded region becomes connected to $p$.

\begin{figure}[H]
\centering
\begin{tikzpicture}

\node (tikzPic) at (0,0) {
\includegraphics[scale=0.57]{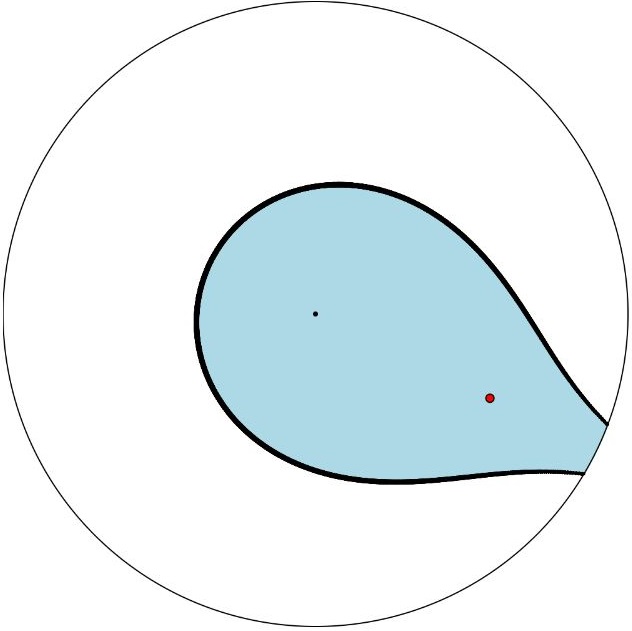}};

\node (Origin) at (0.204, 0.204) {$O$};
\node (Point) at (2.32, -0.77) {$p$};
\node (Ball indicator) at (-1.8, 1.12) {$B (p;R)$};
 \node (diskMarking) at ( -3, 2.93) {$\D$};

\end{tikzpicture}

\caption{The ball of radius $R$ centered at point $p,$ within $\D$ \label{fig:tube}}
    \label{fig:my_label}
\end{figure}
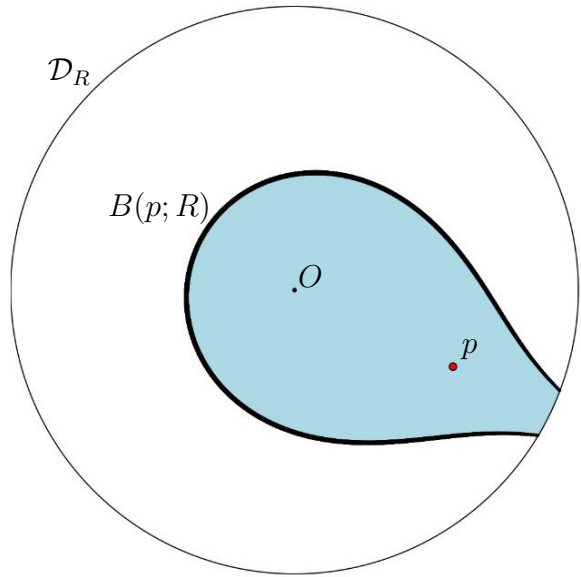

\subsection{Poissonisation of the KPKBV model}
In this paper, we will work on the Poissonisation of the above model. 
Recall that $\D$ was defined to be the disc of hyperbolic radius $R$ around the origin $O$ of the Poincar\'e disc representation of the hyperbolic plane of curvature $-1$.
Here, the vertex set is the point-set of a Poisson point process on $\D$ with intensity
$$n \frac{1}{2\pi}\rho_n (r) dr d\theta .$$
We denote it by $\PPna$. We also denote by $\kappa_{\alpha, \nu, n}$ the Borel measure on $\D$ given by
$$\kappa_{\alpha,\nu, n} (S) = \frac{1}{2\pi} \int_S \rho_n(r) dr d \theta, $$
for any Borel-measurable set $S$. Hence, the number of points that $\PPna$ has inside $S$ is distributed as
$\Po {n\cdot \kappa_{\alpha,\nu, n} (S)}$.
Moreover, the numbers of points in any finite collection of pairwise disjoint Borel-measurable subsets of $\D$ are
independent Poisson-distributed random variables.

We will define the random graph whose vertex set is the set of points of $\PPna$ in $\D$. 
As in $\Gnan$, two vertices/points of $\PPna$ are adjacent if and only if their hyperbolic distance is at most $R$.
We denote the resulting graph by $\Pnan$.


\subsection{The modularity of $\Pnan$} \label{sec:mod}

The notion of modularity was introduced by Newman and Girvan in~\cite{NewmanGirvan}.
For a graph $G=(V,E)$ with $m\geq 1$ edges, define the \emph{modularity score} associated with the 
partition $\A$ of the vertex set $V$ to be
\[ \q_{\A}(G) = \sum_{A \in \A } \left( \frac{e(A)}{m} - \left( \frac{\vol(A)}{2m} \right)^2 \right)
\]
where $e(A)$ denotes the number of edges within part $A$ and $\vol(A) = \sum_{v \in A} {\rm deg}(v)$ denotes the volume of $A$, that is, the sum of the degrees of the vertices in $A$. 

 For graphs $G$ without edges define $\q_{\A}(G)=0$. Note that the definition of modularity extends naturally to weighted graphs and is often used in the weighted form in applications. The term $e(A)$ becomes the sum of the weights of edges in $A$ and the degree of a vertex $\deg(v)$ is the sum of the weights of the edges incident to $v$.

This sum is effectively a comparison between the given network $G$ and a random network with the same degree sequence. The first term $\frac{1}{m} \sum_{A \in \A } e(A)$ is the probability that a randomly chosen edge of $G$ 
will lie inside one of the parts, whereas the term $\sum_{A \in \A} \left(\frac{\vol(A)}{2m} \right)^2$ represents 
the probability that a random edge lies in one of the parts in a uniformly random graph with the same 
degree distribution as $G$. 

On one extreme 
if there were no edges between the parts of $\A$, then 
$\frac{1}{m}\sum_{A \in \A } e(A) =1$. If $\A$ consists of a large number of parts that are comparable in volume, 
then the second term $\sum_{A \in \A}\left(\frac{\vol(A)}{2m}\right)^2$ is small. Hence, such a highly modular partition will have a 
modularity score close to 1. 

With $\cP(V)$ denoting the set of all partitions of $V$ 
the \emph{modularity of graph} $G$ is then 
$$\q(G) = \max \{ \q_{\A}(G) \ : \  \A \in \cP (V)\}.$$ 
 The set $\cP(V)$ includes the trivial partition $\{ V \}$ placing all vertices into the same part. Note that the modularity score of $\{ V\}$ is zero for any graph. Hence for any graph $0\leq \q(G) < 1$ with values near 1 taken to indicate a high level of community structure and values near 0 taken to indicate a lack of community structure. Newman~\cite{ar:Newm2006} determined the modularity of several examples of complex networks, not only social, finding them ranging between 0.3 and 0.8. Among these examples, higher modularity ($>0.7$) was found in the social network of co-authorship among scientists working on condensed matter. 

Brandes et al.~\cite{ar:Brandes2007} showed that finding the modularity of a given graph is NP-hard. {Further it was established by Dinh, Li and Thai that it is NP-hard to approximate modularity to within any constant factor~\cite{dinh2015network}. }
However, community detection in networks has been a central theme in network science. 
Newman~\cite{arNewm2006PNAS} used modularity to design a spectral algorithm for community detection in a given network. A popular algorithm, the Louvain method, is an iterative clustering technique uses the modularity function to compare candidate partitions~\cite{ar:BlondelFast}. 

For binomial random graphs from the $G(n,p)$ model, where on a set of $n$ vertices, each pair is included as an edge independently with probability $p$, there is a transition the typical behaviour of $\q (G(n,p))$ that is determined by $np$. In particular, the third author together with McDiarmid showed~\cite{ERgraphs} that when $np\leq 1+o(1)$, then $\q (G(n,p))$ is concentrated around 1, but when $np$ exceeds and is bounded away from 1, then it scales like $(np)^{-1/2}$.  They have also shown~\cite{REGgraphs} 
that for random $d$-regular graphs of bounded degree, it is bounded away from 0 and 1 with high probability and scales approximately like $1/\sqrt{d}$ when $d$ is large.  
Recently, Lichev and Mitsche~\cite{ar:LubMit2020} showed that for $d=3$ the modularity exceeds $2/3$ (confirming a conjecture of McDiarmid and Skerman) and is below $0.8$ with high probability. They further considered the modularity of random graphs having a given degree sequence with bounded maximum degree.

The main theorem of this paper is that with high probability the modularity of $\Pnan$ is close to 1. 
\begin{theorem} \label{thm:main} 
For any $\alpha > 1/2$ and $\nu >0$, we have 
$$ \q (\Pnan) \to 1,$$
as $n\to \infty$, in probability.
\end{theorem}
As we shall see in the next section, the parameters $\alpha$ and $\nu$ determine the average degree of 
$\Pnan$. In particular, for any given $\alpha > 1/2$, the average degree (of $\G(n;\alpha, \nu)$) turns out to be directly proportional to $\nu$. Unlike the $G(n,p)$ model, the modularity of $\Pnan$ approaches 1 as $n\to \infty$, without any dependence on the average degree or the existence of a giant component.  

\subsubsection*{Notation} We now introduce some notation which we use throughout out proofs.
If $\mathcal{E}_n$ is an event on the probability space
$(\Omega_n, \mathbb{P}_n, \mathcal{F}_n)$, for each $n \in \mathbb{N}$,
we say that $\mathcal{E}_n$ \emph{occurs asymptotically almost surely (a.a.s.)} if $\mathbb{P}_n (\mathcal{E}_n) \rightarrow 1$ as $n \rightarrow \infty$. 
In our context, we will be using the term \emph{a.a.s.} for the sequence of probability spaces of the random 
graphs $\Pnan$.

\section{Typical properties of the KPKBV model}

For $\alpha \in (1/2, \infty)$, Krioukov et al.~\cite{ar:Krioukov} show that the tails of the distribution of the degrees in $\G (n; \alpha, \nu)$  follow a power law with exponent $2\alpha  + 1$.  This was verified rigorously
in~\cite{ar:Gugel}.  
Thus, when $ \alpha \in (1/2, 1)$ the exponent is between 2 and 3.  
There has been experimental evidence that this is indeed the  case in a number of networks arising  in applications (the survey~\cite{BarAlb} contains a comprehensive a list of such examples).
Krioukov et al.~\cite{ar:Krioukov} also observe that the average degree of $\G (n; \alpha, \nu)$ 
is also tuned by the parameter $\nu$ for $\alpha \in (1/2, \infty)$. 
This was proved by Gugelman et al.~\cite{ar:Gugel}. They showed that
the average degree tends to  $8 \alpha^2 \nu / \pi(2\alpha-1)^2$ in probability.
 However, when $\alpha \in (0, 1/2]$, the average degree tends to infinity as $n\to \infty$.
 Thus, in this sense, the regime $\alpha \in (1/2,\infty)$ corresponds to the so-called 
 \emph{thermodynamic regime} in the context of random geometric graphs on the Euclidean plane~\cite{bk:Penrose}.

Gugelman et al.~\cite{ar:Gugel} also showed $\G (n;\alpha, \nu)$ has \emph{clustering coefficient} that is a.a.s. 
bounded away from 0. More precise results about the scaling of the local clustering coefficient in terms of the degrees
of the vertices were obtained by Stegehuis et al.~\cite{ar:CF_vdH2020}. More recently in~\cite{ar:CF2021}, convergence in probability of the clustering coefficient to an explicitly determined constant was derived.

When $\alpha$ is small, there are more points of $\PPna$ near the origin and one may expect increased graph connectivity. The paper
~\cite{BFMgiantEJC} establishes 
that $\alpha = 1$ is the critical point for the emergence of a giant component in $\G (n; \alpha, \nu)$. 
In particular, when $\alpha \in (0,1)$, the fraction of the vertices
contained in the largest component is bounded away from 0 a.a.s.~\cite{BFMgiantEJC}, whereas 
if $\alpha \in (1, \infty)$, the largest component is sublinear in $n$ a.a.s.
For $\alpha = 1$, the component structure depends on $\nu$. If $\nu$ is large enough,
then a giant component exists a.a.s., but if $\nu$ is small enough, then a.a.s. all components have sublinear 
size~\cite{BFMgiantEJC}. 

 The above results were strengthened in~\cite{ar:FM_AAP}. In that paper, it was shown that the fraction of vertices which belong to the largest component converges in probability to a certain constant which depends on $\alpha$ and $\nu$. More specifically, when $\alpha =1$, it turns out that there exists a critical value $\nu_0 \in (0, \infty)$ such that when $\nu$ crosses $\nu_0$ a giant component emerges a.a.s. 
 The papers \cite{KiwiMit} and \cite{KiwiMit2017+} consider the size of the second largest component. 
 Therein, it is shown that when $\alpha \in (0, 1)$ the second largest component has polylogarithmic order a.a.s.

The connectivity of $\G (n; \alpha, \nu)$ was considered by Bode et al. in~\cite{BFMconnRSA}. 
They show that for $\alpha <1/2$ the random graph $\G (n; \alpha, \nu)$ is a.a.s. connected, it is disconnected for $\alpha >1/2$~\cite{BFMconnRSA}.  When $\alpha =1/2$, it turns out that 
the probability of connectivity converges to a certain constant  which is given explicitly in~\cite{BFMconnRSA}.

The a.a.s. disconnectedness of $\G (n; \alpha, \nu)$ for $\alpha > 1/2$ follows easily from the a.a.s. existence of isolated vertices. Recent, asymptotic distributional properties of the number of isolated as well as the extreme points in $\Pnan$ were derived in~\cite{ar:NFJY2020}.  (A point is called extreme, when it is not connected to any other point of larger radius.)
The authors showed that the former satisfies a central limit theorem when $\alpha >1$, but it does not when $\alpha < 1$. However, the number of extreme points satisfies a central limit theorem for any $\alpha >1/2$. This is due to the fact that the number of isolated vertices is sensitive on the existence of a few vertices close to the centre of $\D$. Those a.a.s. appear when $1/2 < \alpha < 1$. 
On the other hand, extreme points have only local dependencies. 

Bounds on the diameter of $\G (n; \alpha, \nu)$ were derived in~\cite{KiwiMit} and~\cite{ar:FriedKrohmerDiam}. 
Therein, polylogarithmic upper bounds on the diameter are
shown. These were improved by M\"uller and Staps~\cite{ar:MullerDiam} who deduced 
a logarithmic upper bound on the diameter.
Furthermore, in~\cite{ar:AbdBodeFound} it is shown that for $\alpha \in (1/2,1)$ the largest component has doubly logarithmic typical distances and it forms what is called an \emph{ultra-small world}. 

\begin{figure}[H]
\minipage{0.32\textwidth}
    \includegraphics[scale = 0.26]{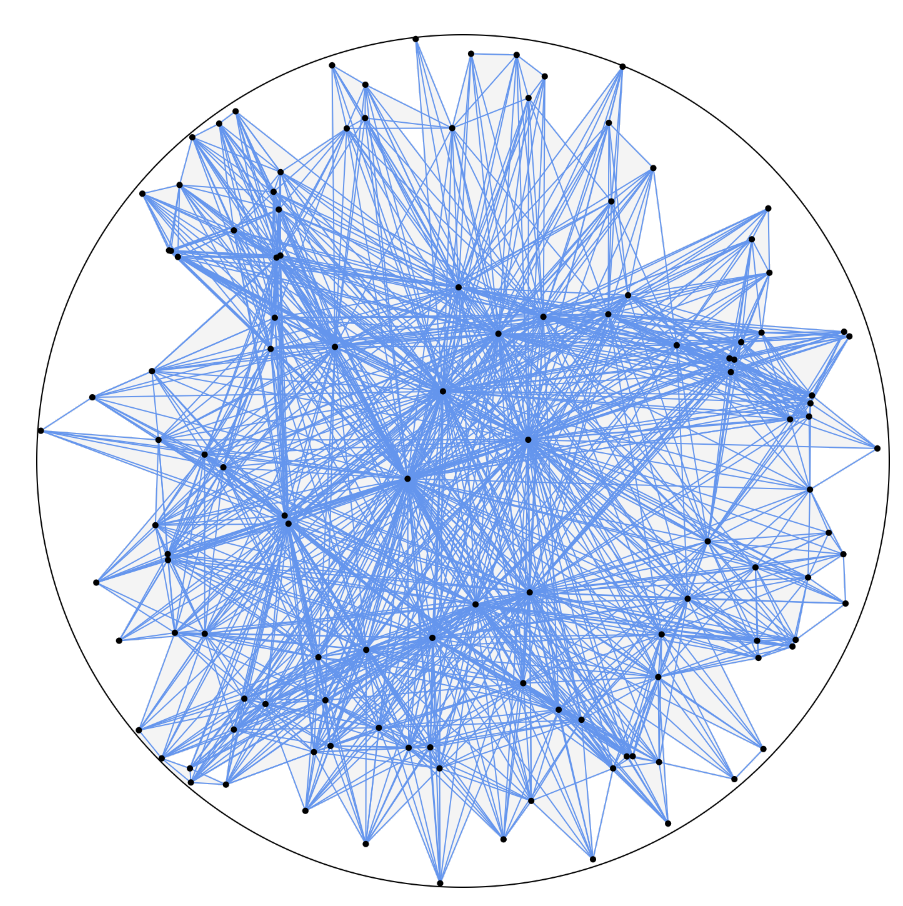}
    \endminipage \hfill
    \minipage{0.32\textwidth}\hspace{-0.3cm}
    \includegraphics[scale = 0.26]{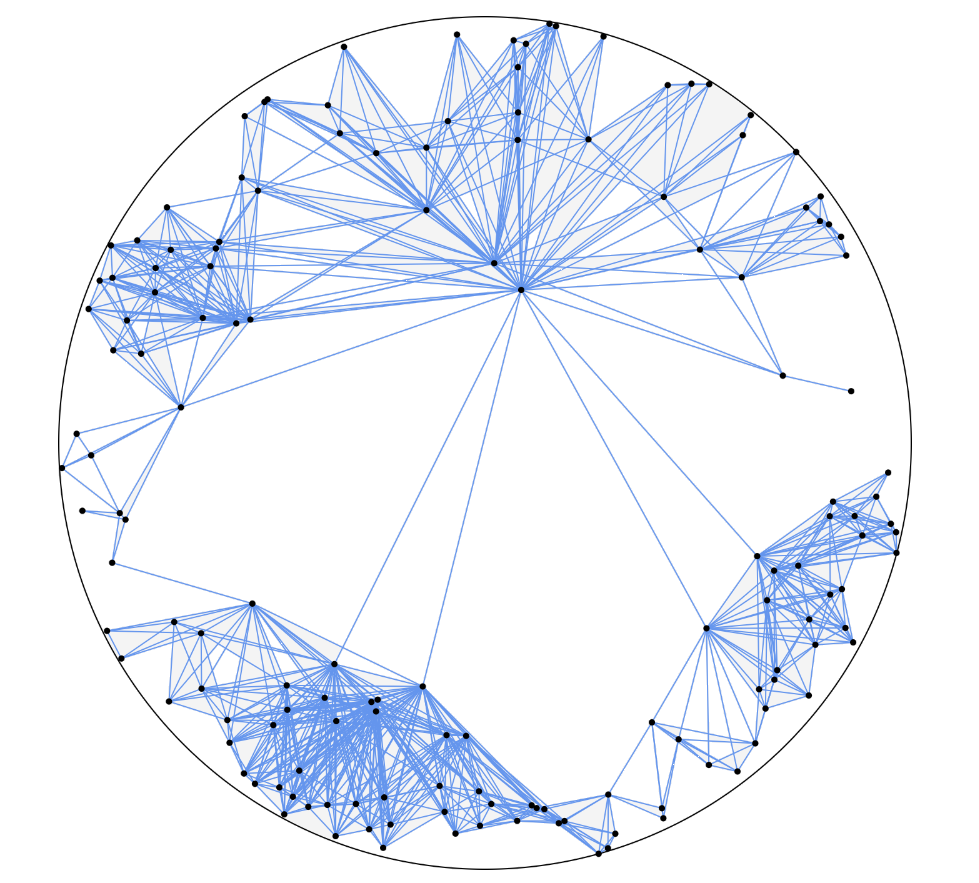}
    \endminipage \hfill
    \minipage{0.32\textwidth}
    \includegraphics[scale = 0.26]{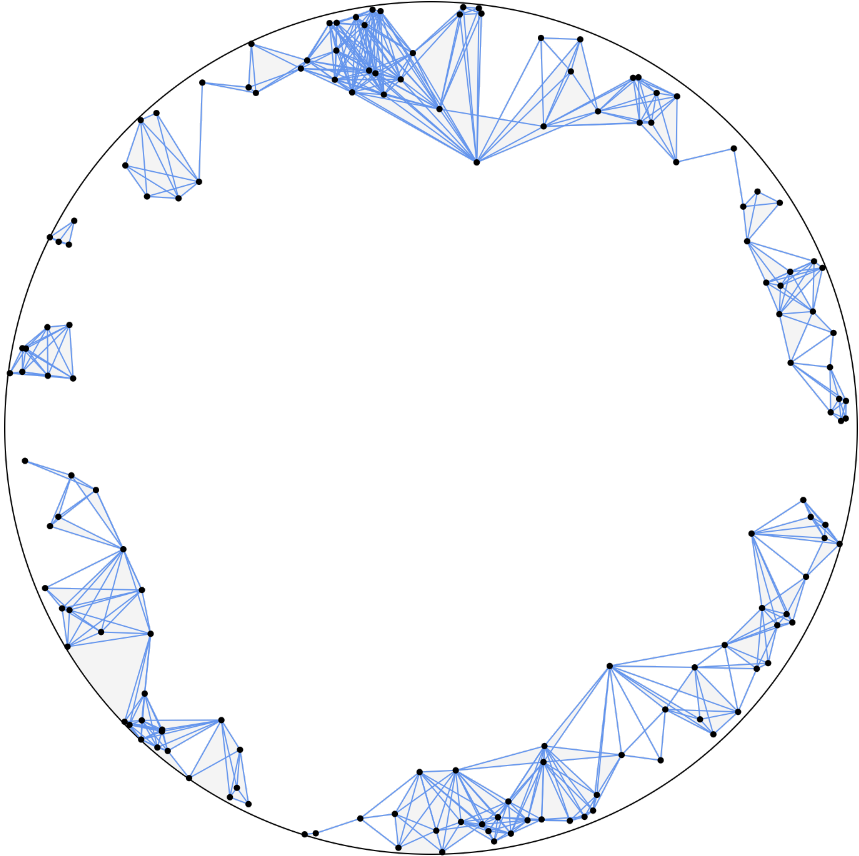}
    \endminipage 
    
    \caption{Three samples of $\G (n; \alpha, \nu)$ with a fixed $n = 150$, $\nu = 2$ and variable $\alpha.$ From left to right: $\alpha = 0.6,$ $\alpha = 1$ and $\alpha = 1.8$}
    
\end{figure}

\subsection{Approximating a ball around a point - geometric notation}

The main lemma in this section provides a useful (almost) characterization of two vertices being within
hyperbolic distance $R$, given their radii.
The lemma reduces a statement about hyperbolic distances to a statement about the relative angle
between two points.
Let us first introduce some notation. For a point $p \in \D$, we let $\theta (p) \in (-\pi, \pi]$ be the angle
$\hat{p O s}$ between $p$ and a (fixed) reference point  $s \in \D$ (moving from $s$ to $p$ in the
anti-clockwise direction). 
For $\theta, \theta' \in (-\pi,\pi]$, we set 
$$|\theta - \theta'|_\pi =  \min \{ |\theta - \theta' |, 2\pi - |\theta - \theta'| \} \in [0,\pi]. $$
For two points $p,p'\in \D$ we denote by $\theta (p,p') \in [0,\pi]$ their relative angle: 
$$\theta (p,p') = |\theta(p) - \theta (p')|_\pi.$$
Also, for $p \in \D$ we let $y(p)$ denote the \emph{defect radius} of $p$ in $\D$. In other words, if $r(p)$ is the radius
(the hyperbolic distance of $p$ from $O$), then $y(p) = R- r(p)$.
The following lemma gives a characterisation of what it is to have hyperbolic distance at most $R$ in terms
of the relative angle between two points. For $r,r'$ such that $r+r'> R$,  
let $\theta_R (r,r') \in (-\pi,\pi]$ be such that 
if two points $p,p'$ with $r(p)=r$ and $r(p')=r'$ have $\theta (p,p') = \theta_R (r,r')$ iff $d_H(p,p') = R$. 
Also, we set $T_R(y,y') = 2 \cdot e^{-R/2} e^{\frac12 (y+y')}$, for $y,y'\in [0,R]$. 

The following lemma is a consequence of Lemma 28 in~\cite{ar:FM_AAP}.
\begin{lemma}\label{lem:relAngle}
Let $\zeta \in (0,1)$. For any $\gamma > 0$ and any $n$ sufficiently large, 
uniformly for any $p, p'\in\D$ with $y(p)+ y(p') \leq \zeta R$ the following holds
$$ \left| \frac{\theta_R (r(p),r(p'))}{T_R (y(p),y(p'))}- 1 \right|< \gamma .$$
\end{lemma}

For a point $p \in \D$, let $B(p;R)$ denote the set of points in $\D$ of hyperbolic distance at most $R$ from 
$p$. We further define 
$$ \check{B}_{\zeta,\gamma} (p) := \{ p' \in \D \ : \  y(p') +y(p)\leq \zeta R, \theta (p,p') < (1+\gamma ) 
T_R(y(p),y(p')) \}.$$ 
Let $\mathcal{A}_r := \D \setminus \mathcal{D}_r$ denote the annulus of the disc $\D$ which consists of all 
points of defect radius at most $R-r$.
The above lemma implies that for any $\zeta \in (0,1)$, $\gamma >0$ and any $n$ sufficiently large we 
have 
\begin{equation} \label{eq:ball_approx}
\check{B}_{\zeta,-\gamma} (p)  \subset B(p;R) \cap \A_{(1-\zeta) R+y(p)} \subset \check{B}_{\zeta, \gamma} (p);
\end{equation}
hence, the set $\check{B}_{\zeta, \gamma} (p)$ includes all points in $B(p;R)$ of defect radius at 
most $\zeta R - y(p)$. 
Further, the following holds and will be useful later on during our second moment calculations. 
\begin{claim}\label{clm:disjoint_balls}
If $\zeta \in (0,1)$ and $\gamma >0$, then for any $n$ sufficiently large
whenever $\theta (p,p') > 4(1 + \gamma) e^{-(1-\zeta) R/2}$ for points $p,p' \in \D$ with 
$y(p),y(p')<R/2$, we have  
$$ \left( B(p;R) \cap \A_{(1-\zeta) R+y(p)}\right) \cap \left(B_R (p') \cap \A_{(1-\zeta) R+y(p')} \right) = \varnothing.$$
\end{claim}
Another result, that will be useful later on, is the bound of the expected number of point of $\PPna$ inside 
$\check{B}_{\zeta, \gamma} (p)$. 
\begin{claim} \label{eq:check_ball_vol}
For any $\zeta \in (1/2,1)$ and $\gamma \in (-1,1)$, uniformly for any $p\in \D$ with $y(p)\leq R/2$ we have 
$$ \E{|\PPna \cap \check{B}_{\zeta, \gamma} (p)|} = \Theta (e^{y(p)/2}). $$
\end{claim}
\begin{proof} 
We calculate
\begin{eqnarray} 
\E{|\PPna \cap \check{B}_{\zeta, \gamma} (p)|}  &=& n \frac{1+\gamma}{2\pi}  \cdot 2
e^{-R/2 + y(p)/2} \cdot \int_{(1-\zeta) R + y(p)}^R 
  e^{(R-\varrho)/2}
\frac{\alpha \sinh (\alpha \varrho)}{\cosh(\alpha R)-1} d \varrho \nonumber \\
&=& \Theta (1) \cdot e^{y(p)/2} \int_{(1-\zeta) R + y(p)}^R e^{(1/2- \alpha) (R-\varrho)} d \varrho \nonumber \\
&=& \Theta (1) \cdot e^{y(p)/2} \int_0^{\zeta R - y(p)} e^{(1/2 -\alpha)y} d y \stackrel{\alpha >1/2}{=} \Theta (e^{y(p)/2}).  \nonumber 
\end{eqnarray}
\end{proof}
Furthermore, since $|\PPna \cap \check{B}_{\zeta, \gamma} (p)|$ follows the Poisson distribution, the above claim
also yields, that for any $\zeta \in (1/2,1)$ and $\gamma >0$,
\begin{equation}  \label{eq:check_ball_vol_sq}
\E {|\PPna \cap \check{B}_{\zeta, \gamma} (p)|^2} = O(e^{y(p)}),
\end{equation}
 uniformly for any $p\in \D$ with $y(p)\leq R/2$.

\subsubsection{Projecting $\D$ onto $\mathbb{R}^2$}

To simplify our calculations, we will transfer our analysis from $\D$ to $\mathbb{R}^2$. In particular, we
will make use of a mapping that was introduced in~\cite{ar:FM_AAP} and reduces our model to a percolation model on $\mathbb{R}^2$. This is achieved using a local approximation of the hyperbolic metric as given in Lemma~\ref{lem:relAngle}.
For a point $p \in \D$, let $(\theta(p), y(p)) \in (-\pi, \pi]\times [0,R]$ denote its angle with respect to a reference point and its defect radius, respectively. 

We define the map $\Phi : \D \to 
\mathcal{B}=(-\frac{\pi}{2}e^{R/2} , \frac{\pi}{2}e^{R/2}] \times [0,R]$,
mapping a point $p=(\theta (p), y(p)) \in \D$ to a point $(x(p),y(p)) \in \mathcal {B}$ 
$$\theta (p) \mapsto x (p):=\frac{1}{2} \theta (p) e^{R/2} \  \mbox{and}\  y(p) \mapsto y(p).$$   
For simplicity, we set
$I : = I (R):=\frac{\pi}{2}e^{R/2}$.

The map $\Phi$ projects the process $\PPna$ to a point process on
$\mathcal{B}$.  

We will approximate this process with the Poisson point process on $\mathcal{B}$
having intensity $$ \frac{2\nu}{\pi} \alpha e^{-\alpha y} dx dy. $$
For any measurable subset $S \subseteq \mathcal{B}$, we
set $\mu_{\alpha,\beta} (S) = \beta \int_S  e^{-\alpha y} dx dy$, with
$\beta = \frac{2\nu \alpha}{\pi}$. 
We denote this Poisson process by $\PP{\alpha}{\beta}$.

The analogue of the relative angle between points in $\D$ is defined as follows.
For $x, x' \in ( -I , I ]$, we let
$$|x-x'|_{\mathcal{B}}:= \min \left\{ |x-x'|, 2I-|x -x'| \right\}. $$

For a positive real number $y<R$, we set $\mathcal{B}(y) := (-\frac{\pi}{2}e^{R/2} , \frac{\pi}{2} e^{R/2}] \times [0,y]$; thus $\mathcal{B}(R) = \mathcal{B}$.
We define the random graph $\Bnan{y}$ with vertex set the point set of $\PP{\alpha}{\beta} \cap \mathcal{B}(y)$, and for any 
distinct $p,p'  \in \PP{\alpha}{\beta}$, the vertices $p,p'$ are adjacent if and only if 
$$|x(p)-x(p')|_{\mathcal{B}} < e^{(y(p)+y(p'))/2}.$$

We define the ball around a point $p\in \mathcal{B}(y)$ as 
$\Ball{p}{y} = \{ p' \in \mathcal{B}(y) \ :  \ |x(p) -x(p')|_{\mathcal{B}}  < e^{\frac12 (y(p) + y(p') )}\}$.
Thus, for a point $p \in \PP{\alpha}{\beta}$, the neighbourhood of $p$ in the random graph 
$\Bnan{y}$ is $\Ball{p}{y} \cap \PP{\alpha}{\beta} \setminus \{p\}$.  Figure~\ref{fig:plane} shows the neighbourhood around a point $p \in \mathcal{B}(y)$. Thus any point lying within the shaded region will be connected to $p.$ The rectangular region bounded by the axis and the dotted line represents a single box in our partition, see Section \ref{sec:modOfB}.

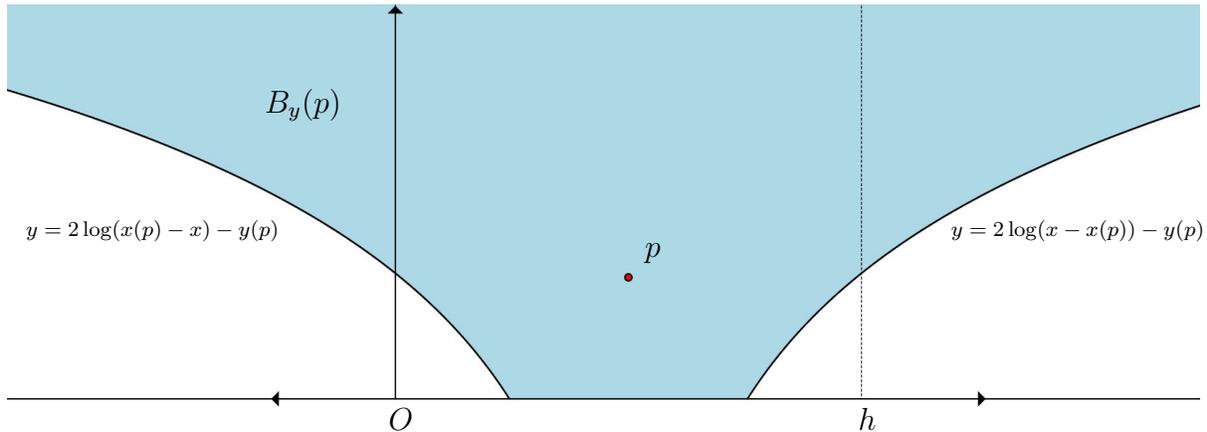
\begin{figure}[H]
\centering
\begin{tikzpicture}

\node (tikzPic) at (0,0) {
\includegraphics[scale = 0.39]{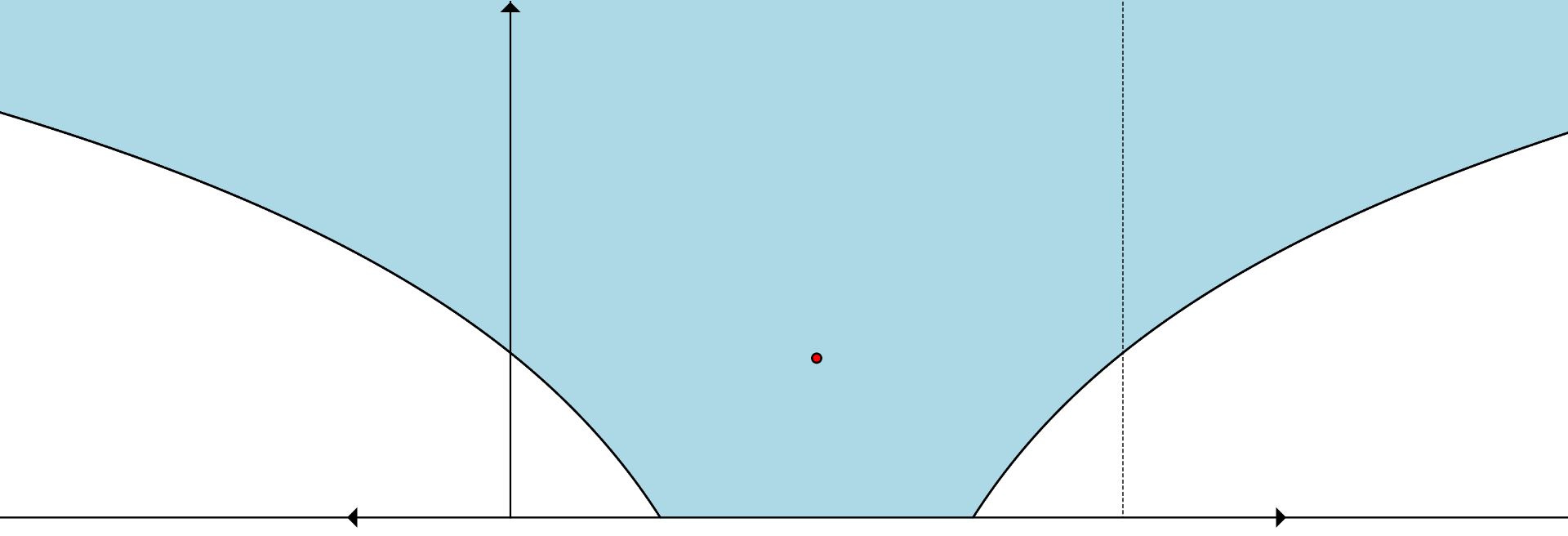}};

\node (Origin) at (-2.7, -2.8) {$O$};
\node (Point) at (0.65, -0.6) {$p$};
\node (Ball indicator) at (-4, 1.38) {$B_{y}(p)$};

\node (leftBranch) at (-6, -0.3) {

\scriptsize{${y = 2\log(x(p) - x) - y(p)}$}};

\node(rightBranch) at (6.3, -0.3){

\scriptsize{$ y = 2\log(x - x(p)) - y(p)$}
};

\node(boundingBox) at (3.5, -2.8) {$h$};

\end{tikzpicture}

\caption{The ball $B_{y}(p).$ \label{fig:plane}}

\end{figure}



\section{Mapping $\Pnan$ into $\mathcal{B}$ and the proof of Theorem~\ref{thm:main}}

To prove Theorem~\ref{thm:main}, it suffices to consider a subgraph of $\Pnan$ which contains most edges 
of it. 
To this end, we use Lemma 5.1 from \cite{ERgraphs}.
\begin{lemma} \label{lem:1st_approx}
Let $G = (V,E)$ be a graph with $|E|\geq 1$, let $E_0$ be a nonempty subset of $E$. 
For $E' = E \backslash E_0$, let $G' =(V,E')$. Then
\[
|\q(G) - \q(G')| < 2 |E_0|/|E|.
\]
\end{lemma}
We will show the following lemma. 
\begin{lemma} \label{lem:cut-out} For every $\eps >0$ there exists $y_{\eps}>0$ such that a.a.s.
$$\vol (\PPna \cap \mathcal{D}_{R-y_\eps})\leq \eps e (\Pnan). $$ 
\end{lemma}
For  a positive real number $r<R$, let $\PBnan{\leq y}$ denote the subgraph of $\Pnan$ induced by the points of $\PPna$ having defect radius at most  $y$. (The subgraph $\PBnan{> y}$ is defined analogously.)
As the number of edges incident to points in $\PPna \cap \mathcal{D}_{R-y_\eps}$ is at most $\vol (\PPna \cap \mathcal{D}_{R-y_\eps})$, 
the above two results imply that for every $\eps >0$ there exists $y_\eps>0$ such that a.a.s. 
$$ |\q (\Pnan)- \q (\PBnan{\leq y_\eps})| < 2 \eps. $$
Thereby, to prove Theorem~\ref{thm:main} it suffices to show that 
\begin{equation}\label{eq:partial_graph} 
\q (\PBnan{\leq y_\eps}) \to 1,
\end{equation}
as $n\to \infty$ in probability. 
To show this, we will couple the random graph $\PBnan{\leq y_\eps}$ with the random graph $\Bnan{y_\eps}$.
\begin{lemma}[Lemmas 27 and 30 in~\cite{ar:FM_AAP}] 
There is a coupling between the point processes $\PP{\alpha}{\beta}$ and $\PPna$ such that 
a.a.s. on the coupling space $\Phi (\PPna) = \PP{\alpha}{\beta}$. 
Furthermore, a.a.s. on the coupling space for any distinct $p,p' \in \PPna$ with $y(p),y(p')\leq R/4$
we have $d_H(p,p')\leq R$ if and only if $\Phi (p') \in \Ball{\Phi (p)}{R}$. 
\end{lemma}
The above lemma implies that there is a coupling between the processes  $\PP{\alpha}{\beta}$ and $\Phi(\PPna)$
on $\mathcal{B}$ such that for any fixed $y>0$ a.a.s., on this coupling space, the two point-sets coincide and moreover the random graph $\Bnan{y}$ is isomorphic to $\PBnan{\leq y}$.  

So we can deduce~\eqref{eq:partial_graph} from the following theorem.
\begin{theorem} \label{thm:mod_approx}
For any $\alpha > 1/2$, $\nu >0$ and any fixed $y>0$, we have 
$$ \q (\Bnan{y}) \to 1,$$
as $n\to \infty$ in probability. 
\end{theorem}

\section{Some general properties of the modularity of a graph}
Let $G=(V,E)$ be a graph. For $A,B \subset V$, let $\bar{A}$ denote $V\setminus A$ and, for disjoint $A,B$, 
let $e(A,B)$ denote the number of edges with 
one endvertex in $A$ and the other in $B$.
It will sometimes be helpful to talk separately of the edge-contribution, also called \emph{coverage} 
$$\q^E_\A(G)=\frac{1}{m} \sum_{A\in\A} e(A)=1-\frac{1}{2m}\sum_{A \in \A} e(A, \bar{A}),$$ and 
the \emph{degree tax} 
$$\q^D_\A(G) =\frac{1}{(2m)^2}\sum_{A \in \A} \vol(A)^2.$$ 
The following lemma provides a lower bound on $\q_{\A}(G)$ with respect to the parameters of a given partition 
$\A$.
\begin{lemma} \label{lem:mod_partition}
Let $G$ be a graph with $m$ edges. Suppose the partition $\A=\{A_1, \ldots, A_k\}$ has the property that for each $1\leq i\leq k$,
\[
e(A_i, \bar{A}_i) \leq \varepsilon m \;\;\;\; \mbox{and} \;\;\;\; |\vol(A_i)  -2m/k| \leq 2m\delta 
\] 
then
\[
\q_{\A}(G) \geq 1 - \frac{k\varepsilon}{2} -\frac{1}{k} - k\delta^2. 
\] 
\end{lemma}

\begin{proof} Define $\delta_i$ to be such that $\vol(A_i) = (1/k + \delta_i)2m$ and note that $\sum_i \delta_i=0$ and by assumption $\forall i$ $|\delta_i| \leq \delta$. We may now bound the degree tax of $\A$,
\[
\q^D_{\A}(G) = \frac{1}{4m^2} \sum_{i=1}^k \vol(A_i)^2 = \sum_{i=1}^k \left(\frac{1}{k}+\delta_i\right)^2 \leq \frac{1}{k}+k\delta^2. 
\]
The edge contribution of $\A$ is $\q_A^E(G)= 1-\sum_i e(A_i, \bar{A}_i)/2m \geq 1 - k\eps /2$ and thus we have our required bound.
\end{proof}

\section{The modularity of $\Bnan{r}$.}\label{sec:modOfB}

In this section, we prove Theorem~\ref{thm:mod_approx}.
We shall make use of the following identity which is an application of the (multivariate)~\emph{Campbell-Mecke} formula (see for example Theorem 4.4~\cite{bk:LastPenrose}):
for a Poisson point process $\mathcal{P}$ on a measurable space $S$ with intensity $\rho$ and a measurable non-negative function $h: S^k \times \mathcal{N} \rightarrow \mathbb{R}$, where $\mathcal{N}$ is the set of all locally finite collections of points in $S$, we have
\begin{equation} \label{eq:Campbell-Mecke}
\begin{split}
& \mathbb{E}  \left(  \sum_{x_1,\ldots, x_k \in \mathcal{P}}^{\neq} h (x_1, \ldots, x_k,
\mathcal{P} \setminus \{x_1,\ldots, x_k\}) \right)\\
& = \int_S \cdots \int_S
\E{h(x_1,\ldots, x_r,\mathcal{P} \cup \{x_1,\ldots, x_k\})} \rho (x_1) \cdots \rho (x_k) dx_1 \cdots dx_k,
\end{split}
\end{equation}
where the sum ranges over all pairwise distinct $k$-tuples of points of $\mathcal{P}$.

Now, we are set to show that for any fixed $y>0$, we have 
$\q (\Bnan{y}) \to 1$ in probability as $n\to \infty$. 
To this end, we will use Lemma~\ref{lem:mod_partition} on a specific partition of the vertex set 
of $\Bnan{y}$. More specifically, we consider a partition of the box $\Bcal_{y}= (-I,I] \times [0,y]$ into $2t$
boxes $B_i:=(i\cdot hI, (i+1) \cdot h I] \times[0,y]$, for $i = - 1/h,\ldots, 1/h-1$, where $h = 1/t$ 
with $t \in \mathbb{N}$.   

Given this partition of the box $\Bcal_{y}$, we let $A_i = \PP{\alpha}{\beta}\cap B_i$, for $i= - t, \ldots, t-1$. 
With $\A = \{A_{-t},\ldots, A_{t-1} \}$, we will show that a.a.s.
\begin{equation} \label{eq:coverage_A} 
\q_{\A} (\Bnan{y})\geq 1 - 4h - o(1).
\end{equation}
Therefore, for $\eps >0$, take $t\in \mathbb{N}$ to be such that $4h =4/t < \eps/2$. So a.a.s. 
$$\q (\Bnan{y}) \geq 1 -\eps. $$

Let us now proceed with the proof of~\eqref{eq:coverage_A}. 
Firstly, note that since the random variables $\vol(A_i)$ are identically distributed, with $m$ denoting the number of edges of the random graph $\Bnan{y}$, we have
\begin{equation} \label{eq:exp_vols_Ai}
\E{\vol (A_i)} = \frac{1}{2t} \Ex{\vol (\PP{\alpha}{\beta} \cap \Bcal_{y} )} = \frac{\E{m}}{t}.
\end{equation}
We will use a second moment argument to show that a.a.s. for each $i = -t ,\ldots, t$, we have 
\begin{equation} \label{eq:vol_conc} 2m \left(\frac{1}{2t} - 3h^2\right) \leq  \vol(A_i) \leq 2m \left(\frac{1}{2t} + 3h^2\right). \end{equation}
Furthermore, we will show that following. 
\begin{claim} \label{clm:out-edges}
There exists a constant $C$ (depending on $y$) such that a.a.s. 
$$\E{e(A_1,\overline{A}_1)} < C. $$
\end{claim}
By the union bound, this implies that for all $i=-t,\ldots ,t-1$ 
$$e(A_i, \overline{A}_i) < \log n.$$
Since a.a.s. $m = \Omega (n)$, we can then 
deduce~\eqref{eq:coverage_A}, applying Lemma~\ref{lem:mod_partition} with $\eps = \log^2 n /n$, 
$\delta = 3h^2$, and $k=t = 1/h$.

We will deduce~\eqref{eq:vol_conc} from Chebyschev's inequality having shown that both the expectation and the variance of $\vol (A_i)$ are of order $n$. 
\begin{claim} \label{clm:var_A} 
We have 
$$\E{\vol (A_1)} = \Theta (n) \ \mbox{and} \  \Var (\vol(A_1) ) = O(n). $$
\end{claim}
Since the random variables $\vol (A_i)$ are identically distributed, the first part of the above claim together 
with ~\eqref{eq:exp_vols_Ai} imply that $\E{m} = \Theta (n)$ too.  Furthermore, Chebyschev's inequality implies that a.a.s. 
$$2 \E{m}\left( \frac{1}{2t} - h^2 \right) \leq  \vol (A_1) \leq 2 \E{m}\left( \frac{1}{2t} + h^2 \right). $$
In turn, the union bound implies that a.a.s. for all $i=-t, \ldots, t-1$, we have 
\begin{equation} \label{eq:vol_conc_i} 
2 \E{m} \left( \frac{1}{2t} - h^2 \right) \leq \vol (A_i) \leq 2 \E{m} \left( \frac{1}{2t} + h^2 \right). 
\end{equation}
Furthermore, a.a.s $m \geq \E{m} (1- h^3 )$. Indeed,  we have by Chebyshev's inequallity that for each $i = -t, \ldots, t-1$

$$\mathbb{P} \left[{\lvert \vol(A_i) - \E{\vol (A_i)} \rvert > h^3 \E{\vol(A_i)}} \right]  \leq \frac{\Var({\vol (A_i)})}{h^6 \E{\vol (A_i)}^2} \stackrel{Claim~\ref{clm:var_A}}{=} o(1). $$

Hence, by the union bound, we have that a.a.s for all $i = -t, \ldots, t-1$ that $\vol (A_i) \geq (1-h^3)\E{\vol (A_i)}.$ Therefore by the Handshaking Lemma, $\sum_{i=-t}^{t-1} \vol (A_i) = 2m$ whereby

$$2m = \sum_{ -t \leq i \leq t-1}\vol (A_i) \geq  \sum_{-t \leq i \leq t-1} (1-h^3)\E{\vol (A_i)} =2 \E{m} (1-h^3)$$
and 
$$2m = \sum_{ -t \leq i \leq t-1}\vol (A_i) \leq  \sum_{-t \leq i \leq t-1} (1+h^3)\E{\vol (A_i)} =2 \E{m} (1+h^3).$$ 

From the above, we deduce~\eqref{eq:vol_conc} since a.a.s. for all $i=-t,\ldots, t-1$  

\begin{equation*} \label{eq:vol_conc_actual} 
\vol (A_i) \leq \left( \frac{1}{2t} + h^2 \right) (1-h^3)^{-1} m \leq \left( \frac{1}{2t} + h^2 \right) (1+h^2) m  \leq 
\left( \frac{1}{2t} + 3h^2 \right) m,
\end{equation*}
provided that $t\geq 2$  (so that $h^2 >  h^3 + h^5$ which is equivalent to $1 > h + h^3$ and holds 
if $h \leq 1/2$), and 
\begin{equation*} \label{eq:vol_conc_actual_low} 
\vol (A_i) \geq \left( \frac{1}{2t} - h^2 \right) (1+h^3)^{-1} m \geq \left( \frac{1}{2t} - h^2 \right) (1-h^3) m  \geq 
\left( \frac{1}{2t} - 3h^2 \right) m.
\end{equation*}
. 
\begin{proof}[Proof of Claim~\ref{clm:out-edges}]
Firstly, let us point out that if a point $p\in A_1$ is far from the boundary of $B_1$, then it does not contribute to 
$e(A_1,\overline{A}_1)$. To quantify this, let us recall that for another  $p' \in 
\mathcal{B}(y)$, if $|x(p') - x(p)|_{\mathcal{B}} > e^{\frac12 (y(p)+y)}$, then $p' \not \in \Ball{p}{y}$. Since, $y(p)\leq y$ too, we can further conclude that for any point $p' \in \mathcal{B}(y)$, 
if $|x(p') - x(p)|_{\mathcal{B}} > e^y$, then $p' \not \in \Ball{p}{y}$. 

Hence, the only points $p\in A_1$ that may contribute to $e(A_1,\overline{A}_1)$ are such that 
$0\leq x(p)< e^{y}$ or $hI - e^{y} \leq x(p) < hI$. Let $A_1^{(1)}$ denote the set of the former and 
$A_1^{(2)}$ the set of the latter.  
Hence, 
\begin{eqnarray*} 
\E{e(A_1,\overline{A}_1)} \leq \E{\vol(A_1^{(1)})} + \E{\vol (A_1^{(2)})} = 2 \cdot \E {\vol (A_1^{(1)})}, 
\end{eqnarray*}
where the last equality holds since the random variables $\vol(A_1^{(1)})$ and $\vol(A_1^{(2)})$ are identically 
distributed. For a finite set of points $P$ and a point $p \in P$, 
we let $\deg (p;P) = | \Ball{p}{y} \cap P|$. 
Now, we apply the Campbell-Mecke formula~\eqref{eq:Campbell-Mecke} and get 
\begin{eqnarray*}
\lefteqn{\E{\vol(A_1^{(1)})} = \E {\sum_{p \in \PP{\alpha}{\beta} \cap A_1^{(1)}} \deg(p;\PP{\alpha}{\beta})}} \nonumber \\
&\stackrel{\eqref{eq:Campbell-Mecke}}{=}& 
\beta \cdot \int_0^{y}\int_0^{e^{y}} \E{\deg ((x_0,y_0)); \PP{\alpha}{\beta} 
\cup \{(x_0,y_0)\} } \cdot e^{-\alpha y_0}dx_0 dy_0. 
\end{eqnarray*}
But $\E{\deg ((x_0,y_0)); \PP{\alpha}{\beta} \cup \{(x_0,y_0)\} } = |\Ball{(x_0,y_0)}{y} \cap \PP{\alpha}{\beta}| 
= O(1)$, uniformly over 
all $x_0\in (0,e^{y}]$ and $y_0\in [0,y]$. 
So 
\begin{equation*}
\E{\vol(A_1^{(1)})} =O(1) \cdot \int_0^{y}\int_0^{e^{y}} e^{-\alpha y_0}dx_0 dy_0 = O(1).
\end{equation*}
\end{proof}

\begin{proof}[Proof of Claim~\ref{clm:var_A}] 
We will calculate $\E{\vol (A_1)}$ with the use of the Campbell-Mecke formula~\eqref{eq:Campbell-Mecke}: 
\begin{eqnarray}
\E{\vol (A_1)} &=& \E{ \sum_{p \in A_1 \cap \PP{\alpha}{\beta}} \deg (p; \PP{\alpha}{\beta} \cup \{p\}) }  \nonumber \\
&=&\beta \cdot \int_{0}^{hI} \int_0^y  \E{\deg ((x_0,y_0); \PP{\alpha}{\beta} \cup \{(x_0,y_0)\})} e^{-\alpha y_0}dy_0 dx_0 \nonumber \\ 
&=& \beta hI \cdot \int_0^y  \E{\deg ((0,y_0); \PP{\alpha}{\beta} \cup \{(0,y_0)\})} e^{-\alpha y_0}dy_0
\label{eq:vol_A1_first}
\end{eqnarray}
since $\PP{\alpha}{\beta}$ is homogeneous on the $x$-coordinate and $\deg ((x_0,y_0); \PP{\alpha}{\beta} \cup \{(x_0,y_0)\})$ are identically distributed with respect to $x_0$. 
Now, 
\begin{eqnarray}
\E{\deg ((0,y_0); \PP{\alpha}{\beta} \cup \{(0,y_0)\})}  &=& 2\beta \cdot \int_{0}^{y} e^{(y_0 + y_0')/2} e^{-\alpha y_0'} d y_0' \nonumber \\
&\stackrel{\alpha > 1/2}{=}& \frac{2\beta}{\alpha - 1/2} e^{y_0/2}\left( 1- e^{-y(\alpha - 1/2)} \right).\nonumber 
\end{eqnarray}
We substitute the integrand in~\eqref{eq:vol_A1_first} with the above expression and get
\begin{eqnarray} 
\E{\vol (A_1)} &=& hI  \frac{2\beta^2}{\alpha - 1/2} \left( 1- e^{-y(\alpha - 1/2)} \right)
\cdot \int_0^y e^{y_0/2 -\alpha y_0}dy_0 \nonumber \\
&=& 2hI  \left[\frac{\beta}{\alpha - 1/2} \left( 1- e^{-y(\alpha - 1/2)} \right)\right]^2 = \Theta (n). \nonumber
\end{eqnarray}

Now, we will calculate $\Var (\vol(A_1)) = \E{\vol(A_1)^2} - (\E{\vol(A_1)})^2$, again with the use of the Campbell-Mecke formula~\eqref{eq:Campbell-Mecke}. 
We write 
\begin{eqnarray}
&&\E{\vol(A_1)^2} = \E {\sum_{p,p' \in \PP{\alpha}{\beta} \cap B_1} \deg(p; \PP{\alpha}{\beta}) \cdot 
\deg(p';\PP{\alpha}{\beta})}\stackrel{\eqref{eq:Campbell-Mecke}}{=} \nonumber \\
&& 
\int_0^{y}\int_0^{h I} \int_0^{y} \int_0^{hI} \E{\deg ((x_0,y_0);\PP{\alpha}{\beta} 
\cup \{(x_0,y_0),(x_0',y_0')\}) \cdot \deg ((x_0',y_0'); \PP{\alpha}{\beta} 
\cup \{(x_0,y_0),(x_0',y_0')\}) } \times \nonumber \\
& & \hspace{2.5cm} e^{-\alpha y_0} e^{-\alpha y_0'} dx_0' dy_0' dx_0 dy_0. \label{eq:2ndmom_inter}
\end{eqnarray}
We will now argue that for the majority of the pairs of points $(x_0,y_0), (x_0',y_0') \in B_1$, the expectation that is
inside this integral factorises.
Suppose without loss of generality that $x_0 < x_0'$. In this case, 
$B_y((x_0,y_0)) \cap B_y((x_0',y_0')) = \varnothing$ if and only if $x_0' - x_0 > e^{(y_0'+y)/2} + e^{(y_0+y)/2}$. 
So, if this is the case, the random variables $\deg ((x_0,y_0);\PP{\alpha}{\beta} 
\cup \{(x_0,y_0),(x_0',y_0')\})$ and  $\deg ((x_0',y_0'); \PP{\alpha}{\beta} 
\cup \{(x_0,y_0),(x_0',y_0')\})$ are independent. 

For given $y_0,y_0'\in [0,y]$, we let 
$$S(y_0,y_0') = \{(x_0,x_0')\in (0,hI]\times (0,hI] \ : \ 0< x_0'-x_0 \leq e^{(y_0'+y)/2} + e^{(y_0+y)/2}\}.$$
With this definition, we split the quadruple integral in~\eqref{eq:2ndmom_inter} in the following way: 
\begin{eqnarray}
& &\int_0^{y}\int_0^{y} \int_{(0,hI]\times (0,hI]\setminus S(y_0,y_0')} 
\E{\deg ((x_0,y_0)) \cdot \deg ((x_0',y_0')); \PP{\alpha}{\beta} 
\cup \{(x_0,y_0),(x_0',y_0')\}}  \times \nonumber \\
& & \hspace{2.5cm} e^{-\alpha y_0} e^{-\alpha y_0'} d x_0 dx_0'  dy_0 dy_0' \nonumber \\
&+& \int_0^{y}\int_0^{y} \int_{S(y_0,y_0')} 
\E{\deg ((x_0,y_0)) \cdot \deg ((x_0',y_0')); \PP{\alpha}{\beta} 
\cup \{(x_0,y_0),(x_0',y_0')\}}  \times \nonumber \\
& & \hspace{2.5cm} e^{-\alpha y_0} e^{-\alpha y_0'} d x_0 dx_0'  dy_0 dy_0'. \label{eq:split}
\end{eqnarray}
If $(x_0,x_0') \in (0,hI]\times (0,hI]\setminus S(y_0,y_0')$, then 
the random variables $\deg ((x_0,y_0);\PP{\alpha}{\beta} \cup \{(x_0,y_0),(x_0',y_0') \})$ and $\deg ((x_0',y_0');\PP{\alpha}{\beta} \cup \{(x_0,y_0),(x_0',y_0') \})$ are independent. 
In the first integral, the integrand is 
\begin{eqnarray*} 
\lefteqn{\E{\deg ((x_0,y_0)) \cdot \deg ((x_0',y_0')); \PP{\alpha}{\beta} 
\cup \{(x_0,y_0),(x_0',y_0')\}}=}  \\
& & \E{\deg ((x_0,y_0)); \PP{\alpha}{\beta} \cup \{(x_0,y_0),(x_0',y_0')\}} \cdot 
\E{\deg ((x_0',y_0')); \PP{\alpha}{\beta}  \cup \{(x_0,y_0),(x_0',y_0')\}} \\
&=& \E{\deg ((x_0,y_0)); \PP{\alpha}{\beta} \cup \{(x_0,y_0)\}} \cdot 
\E{\deg ((x_0',y_0')); \PP{\alpha}{\beta}  \cup \{(x_0',y_0')\}}.
\end{eqnarray*}
Therefore, we can bound the first integral in~\eqref{eq:split} as follows: 
\begin{eqnarray*} 
& &\int_0^{y}\int_0^{y} \int_{(0,hI]\times (0,hI]\setminus S(y_0,y_0')} 
\E{\deg ((x_0,y_0)) \cdot \deg ((x_0',y_0')); \PP{\alpha}{\beta} 
\cup \{(x_0,y_0),(x_0',y_0')\}}  \times \nonumber \\
& & \hspace{2.5cm} e^{-\alpha y_0} e^{-\alpha y_0'} d x_0 dx_0'  dy_0 dy_0' \\
&=&\int_0^{y}\int_0^{y} \int_{(0,hI]\times (0,hI]\setminus S(y_0,y_0')} 
 \E{\deg ((x_0,y_0)); \PP{\alpha}{\beta} \cup \{(x_0,y_0)\}} \cdot 
\E{\deg ((x_0',y_0')); \PP{\alpha}{\beta}  \cup \{(x_0',y_0')\}}  \times \nonumber \\
& & \hspace{2.5cm} e^{-\alpha y_0} e^{-\alpha y_0'} d x_0 dx_0'  dy_0 dy_0' \nonumber \\
&\leq&\int_0^{y}\int_0^{y} \int_{(0,hI]\times (0,hI]} 
 \E{\deg ((x_0,y_0)); \PP{\alpha}{\beta} \cup \{(x_0,y_0)\}} \cdot 
\E{\deg ((x_0',y_0')); \PP{\alpha}{\beta}  \cup \{(x_0',y_0')\}}  \times \nonumber \\
& & \hspace{2.5cm} e^{-\alpha y_0} e^{-\alpha y_0'} d x_0 dx_0'  dy_0 dy_0' \nonumber \\
&=& 
\left(\int_0^{y} \int_0^{hI} 
 \E{\deg ((x_0,y_0)); \PP{\alpha}{\beta} \cup \{(x_0,y_0)\}} e^{-\alpha y_0} d x_0 dy_0   \right)^2.
\end{eqnarray*}
But by the Campbell-Mecke formula~\eqref{eq:Campbell-Mecke}, the latter is 
$$(\E{\vol(A_1)})^2 =\left(\int_0^{y} \int_0^{hI} 
 \E{\deg ((x_0,y_0)); \PP{\alpha}{\beta} \cup \{(x_0,y_0)\}} e^{-\alpha y_0} d x_0 dy_0   \right)^2.$$
 
 Now, let us consider the second integral in~\eqref{eq:split}. 
 In this case, note that uniformly for every $y_0,y_0' \in [0,y]$ and $(x_0,x_0') \in S(y_0,y_0')$, we have 
 $$\E{\deg ((x_0,y_0)) \cdot \deg ((x_0',y_0')); \PP{\alpha}{\beta} 
\cup \{(x_0,y_0),(x_0',y_0')\}} = O(1).$$ Therefore, 
\begin{eqnarray*}
&&\int_0^{y}\int_0^{y} \int_{S(y_0,y_0')} 
\E{\deg ((x_0,y_0)) \cdot \deg ((x_0',y_0')); \PP{\alpha}{\beta} 
\cup \{(x_0,y_0),(x_0',y_0')\}}  \times \nonumber \\
&& \hspace{2.5cm} e^{-\alpha y_0} e^{-\alpha y_0'} d x_0 dx_0'  dy_0 dy_0' \nonumber \\
&=&O(1) \cdot \int_0^{y}\int_0^{y} \int_{S(y_0,y_0')} 
e^{-\alpha y_0} e^{-\alpha y_0'} d x_0 dx_0'  dy_0 dy_0' \nonumber \\ 
&=&O(1) \cdot \int_0^{y}\int_0^{y} \int_0^{hI} \int_{x_0-2e^{y}}^{x+2e^{y}}
e^{-\alpha y_0} e^{-\alpha y_0'} d x_0 dx_0'  dy_0 dy_0'  \nonumber \\
&=& O(1) \int_0^{y}\int_0^{y} \int_0^{hI} e^{-\alpha y_0} e^{-\alpha y_0'} dx_0'  dy_0 dy_0' \nonumber \\
&=& O (n).
\end{eqnarray*}
Thus, we conclude that 
$$ \E{\vol (A_1)^2} \leq (\E{\vol (A_1)})^2 + O(n),$$
whereby 
$$\Var (\vol (A_1)) = O(n). $$
\end{proof}

\section{Proof of Lemma~\ref{lem:cut-out}}

Here, we return to the probability space associated with the random graph $\PBnan{}$. In particular, we 
will work with a subset of the point process $\PPna$ on $\D$, which we denote by $\PPnan$: we set 
$\PPnan = \PPna\setminus  \mathcal{D}_{\delta R}$, for some $\delta \in (0,1)$.   
In other words, $\PPnan$ is $\PPna$ but without the points inside the disc $\mathcal{D}_{\delta R}$.  
The reason for working with this process is that it is hard to bound the degrees of the points of $\PPna$ which 
may appear close to the centre of $\D$. However, we can show that the two processes coincide a.a.s. provided that $\delta$ is small enough. 
\begin{claim} \label{clm:empty_centre}
If $\delta < 1 - 1/(2\alpha)$, then a.a.s. 
$$\PPnan = \PPna.$$
\end{claim}
\begin{proof} 
This follows from a simple first moment argument. Indeed, 
$$\E{|\PPna \cap \mathcal{D}_{\delta R}|} =  n \cdot \kappa_{\alpha,\nu, n} (\mathcal{D}_{\delta R}) 
= n \cdot \frac{1}{2\pi} \int_{-\pi}^{\pi} \int_{0}^{\delta R} \rho_n(r) dr d \theta. $$
But 
$$ \int_{0}^{\delta R} \rho_n(r) dr = \int_{0}^{\delta R} \frac{\alpha \sinh (\alpha r)}{\cosh (\alpha R) -1} dr = \frac{\cosh (\alpha \delta R) -1}{\cosh (\alpha R)-1} = O(n^{-2\alpha (1-\delta)}). $$
Therefore, 
$$ \E{|\PPna \cap \mathcal{D}_{\delta R}|} = O(n^{1-2\alpha (1-\delta)}).$$ 
So, if $\delta < 1 - 1/(2\alpha)$, then the exponent is negative and this expected value is $o(1)$. 
\end{proof}
Note that $1-1/(2\alpha) <1$, as $\alpha > 1/2$. Furthermore, note that the definition of $\PPnan$ allows for both processes to be defined on the same probability space, thus being naturally coupled. 
The intensity measure of $\PPnan$ is $n \cdot \kappa_{\alpha, \nu, n} (\cdot \setminus \mathcal{D}_{\delta R})$. 
For the moment, we shall assume that $\delta < 1 - 1/(2\alpha)$, so that the conclusion of Claim~\ref{clm:empty_centre} holds. 

Now, for a point $p \in \D$ and a finite set of points $P\subset \D$, we set 
$\deg (p;P) = |B(p;R) \cap P|$.   
For $0\leq y_1 < y_2 \leq R$, let $\A_{y_1,y_2}\subset \D$ denote the annulus inside $\D$ consisting of those points in $\D$ having defect radius between $y_1$ and $y_2$. 
We set 
$$ X_{y_1,y_2} (P) =  \sum_{p \in P\cap \A_{y_1,y_2}} \deg (p;P).$$

Clearly, for any $0< y< R$ on the event $\{\PPna = \PPnan\}$ we have
$$\vol (\PPna \cap \mathcal{D}_{R-y}) = X_{y,R} (\PPnan) $$ 
and 
$$ e(\PBnan{}) = \frac{1}{2} X_{0,R}(\PPnan). $$
So, on $\{\PPna = \PPnan\}$, if $\vol (\PPna \cap \mathcal{D}_{R-y})   > \eps e (\PBnan{})$, for some $\eps>0$, then 
\begin{equation} \label{eq:transfer_to_X}
X_{y,R} (\PPnan) > \frac{\eps}{2} X_{0,R} (\PPnan). 
\end{equation}
We will give a general result on the concentration of the 
sum $X_{y,R} (\PPna)$, parametrised by $y$. 
We will show the following. 
\begin{lemma} \label{lem:X_conc}
For any fixed $y\geq 0$, we have 
$$ \frac{ X_{y,R} (\PPnan)}{\E{X_{y,R} (\PPnan)}} \to 1, $$
as $n\to \infty$ in probability.  
\end{lemma}
Furthermore, we show that $\E{X_{y,R} (\PPnan)}$ decays exponentially in $y$. 
\begin{lemma} \label{lem:expec_exp_decay} 
For any $0 \leq y < R/4$ and any $n$ sufficiently large, we have 
$$\E{X_{y,R} (\PPnan)} \leq 2 e^{-(\alpha -1/2) y} \cdot \E{X_{0,R} (\PPnan)}. $$
\end{lemma}
The above two lemmas imply that a.a.s. 
$$X_{y,R} (\PPnan) \leq 3 e^{- (\alpha -1/2)y} X_{0,R} (\PPnan).$$
If we set $y = y_\eps := \frac{1}{\alpha -1/2} \cdot \log (6/\eps)$, it follows from~\eqref{eq:transfer_to_X} that 
\begin{eqnarray*} 
\mathbb{P} (e (\PBnan{>y_\eps})  > \eps e (\PBnan{})) &=&
\mathbb{P} (X_{y_\eps,R} (\PPna) > \frac{\eps}{2} X_{0,R} (\PPnan) ) \\
&=& \mathbb{P} (X_{y_\eps,R} (\PPna) > 3 e^{- (\alpha -1/2)y_\eps} X_{0,R} (\PPnan) ) = o(1). 
\end{eqnarray*}
This concludes the proof of Lemma~\ref{lem:cut-out}, assuming Lemmas~\ref{lem:X_conc} and~\ref{lem:expec_exp_decay}.

We now proceed with the proofs of these two lemmas. 
\begin{proof}[Proof of Lemma~\ref{lem:expec_exp_decay}]

We begin with an upper bound on the expected value of $X_{y,R}(\PPnan)$. 
Note that for $S < R$ we have $X_{y,S}(\PPnan) \leq X_{y,R} (\PPnan)$.
So we can bound
\begin{eqnarray} 
0 &\leq&  X_{y,R} (\PPnan) - X_{y,R/2} (\PPnan) \leq 
2 \cdot | \{p \in \PPnan \cap \mathcal{D}_{R/2}\}|^2 \nonumber \\
&& \hspace{7cm} + \sum_{p\in \PPnan \cap \A_{R/2, (1-\delta)R}} \deg (p;\PPnan \cap \A_{0,R/2}). \nonumber \\ 
& & \label{eq:diff_decomp}
\end{eqnarray}
We will show that the right-hand side is essentially sub-linear. 
\begin{claim} \label{clm:1st_approx}
 $\E{X_{y,R} (\PPnan)} -\E{X_{y,R/2} (\PPnan)} = o(n).$
\end{claim}
\begin{proof}
The expected value of the first term on the right hand side of~\eqref{eq:diff_decomp} is 
\begin{eqnarray*} 
\E{ | \{p \in \PPnan \cap \mathcal{D}_{R/2}\}|} &=& 
n \cdot \kappa_{\alpha, \nu,  n} (\mathcal{D}_{R/2}) < n\cdot  \frac{\alpha}{2\pi} \int_{0}^{R/2} \int_{-\pi}^{\pi} \frac{\sinh (\alpha r)}{\cosh (\alpha R)-1} d \theta d r \\
&=& n\cdot  \frac{\cosh (\alpha R/2) - 1}{\cosh (\alpha R)-1} = O(n^{1-\alpha}) \stackrel{\alpha >1/2}{=} o(n^{1/2}).
\end{eqnarray*}
Since this random variable is Poisson-distributed, the expected value of its square is proportional to the 
square of its expected value. Thereby, 
\begin{equation} \label{eq:sum1} 
\E{| \{p \in \PPnan \cap \mathcal{D}_{R/2}\}|^2} = o(n). 
\end{equation}

We now bound the expected value of the last term in~\eqref{eq:diff_decomp}, 
using the Campbell-Mecke formula~\eqref{eq:Campbell-Mecke}:
\begin{eqnarray}
&& \E {\sum_{p\in \PPnan \cap \A_{R/2, (1-\delta)R}} \deg (p;\PPnan \cap \A_{0,R/2})} =  \nonumber \\
&& n \cdot \frac{1}{2\pi} \int_{-\pi}^{\pi} \int_{\delta R}^{R/2} \E{\deg ((\varrho, \theta); (\PPnan \cup \{(\varrho,\theta)\}) \cap \A_{0,R/2})} \rho_n (\varrho) d \varrho d \theta. \label{eq:term2}
\end{eqnarray}
For a point $p=(\varrho,\theta) \in \D$ (here $\varrho$ is the radius of $p$), we set $h_\zeta (p) :=\zeta R - R+ \varrho$. 
We will use the upper bound which is a consequence of~\eqref{eq:ball_approx}: for $\zeta \in (0,1)$ and 
$\gamma \in (0,1)$ and for $n$ sufficiently large 
\begin{equation} \label{eq:deg_decomp_upper} \deg ((\varrho, \theta); (\PPnan \cup \{(\varrho,\theta)\}) \cap \A_{0,R/2}) \leq 
| \check{B}_{\zeta,\gamma} ((\varrho,\theta)) \cap \A_{0,R/2} \cap \PPnan| +
|B^{\uparrow} ((\varrho,\theta);\PPnan)|, 
\end{equation}
where
$$ B^{\uparrow} ((\varrho,\theta);P):= \{p \in P \ : \ y(p) > h_\zeta ((\varrho, \theta)) \}.$$
Thereby, 
\begin{eqnarray*} 
\lefteqn{\E{\deg ((\varrho, \theta); (\PPnan \cup \{(\varrho,\theta)\}) \cap \A_{0,R/2})} \leq}\\
& &\hspace{2cm} \E{| \check{B}_{\zeta,\gamma} ((\varrho,\theta)) \cap \A_{0,R/2} \cap \PPnan|} + \E{|B^{\uparrow} ((\varrho,\theta);\PPnan)|}.
\end{eqnarray*}
Now, the first term on the right hand side of~\eqref{eq:deg_decomp_upper} can be bounded as follows: 
\begin{eqnarray}
\E{|\check{B}_{\zeta,\gamma} ((\varrho,\theta)) \cap \A_{0,R/2} \cap \PPnan|} &=& n \frac{1+\gamma}{2\pi}  \cdot 2
e^{-R/2 + (R-\varrho)/2} \cdot \int_{R/2}^R e^{(R-z)/2}
\frac{\alpha \sinh (\alpha z)}{\cosh(\alpha R)-1} d z \nonumber \\
&=& \Theta (1) \cdot e^{(R-\varrho)/2} \int_{R/2}^R e^{(1/2- \alpha) (R-z)} d z \nonumber \\
&=& \Theta (1) \cdot e^{(R-\varrho)/2} \int_0^{R/2} e^{(1/2 -\alpha)z} dz = \Theta (e^{(R-\varrho)/2}).  \nonumber 
\end{eqnarray}
Note that this bound is uniform over all $\varrho \in (\delta R , R/2)$.
Substituting it in~\eqref{eq:term2} we get 
\begin{eqnarray}
 \E {\sum_{p\in \PPnan \cap \A_{R/2, (1-\delta)R}}
 |\check{B}_{\zeta,\gamma} ((\varrho,\theta)) \cap \A_{0,R/2} \cap \PPnan|} &=&  O(n)  \int_{-\pi}^{\pi} \int_{\delta R}^{R/2} e^{(R-\varrho)/2} \rho_n (\varrho) d \varrho d \theta \nonumber \\ 
&=& O(n)  \int_{\delta R}^{R/2} e^{(R-\varrho)/2} e^{-\alpha (R-\varrho)} d \varrho \nonumber \\
&\stackrel{\alpha >1/2}{=}& O(n) e^{-(\alpha -1/2) R/2} = o(n). 
\end{eqnarray}

For the second term we have:
\begin{eqnarray} 
\lefteqn{\E{|B^{\uparrow} ((\varrho,\theta);\PPnan)|} = 
n \cdot \kappa_{\alpha, n} ( |\{p \ : \ y(p) > \zeta R - R+\varrho \}|)} \nonumber \\ 
&=& n\cdot  \frac{\alpha}{2\pi} \int_{0}^{2R - \zeta R - \varrho} 
\int_{-\pi}^{\pi} \frac{\sinh (\alpha r)}{\cosh (\alpha R)-1} d \theta d r \nonumber \\
&=& n\cdot  \frac{\cosh (\alpha (2R - \zeta R -\varrho)) - 1}{\cosh (\alpha R)-1} 
= O(n \cdot e^{\alpha (R(1-\zeta) -\varrho)}), \label{eq:upper_ball_mass}
\end{eqnarray}
uniformly over all $R/2 < \varrho < R-y$. 
Therefore, 
\begin{eqnarray}
\lefteqn{n \cdot \frac{1}{2\pi} 
\int_{\delta R}^{R/2} \int_{-\pi}^\pi \E{ |B^{\uparrow} ((\varrho,\theta);\PPnan) | ) } \rho_n (\varrho) 
d\theta d\varrho  =} \nonumber \\
& & O(n^2) \cdot e^{\alpha R (1-\zeta)} \cdot \int_{\delta R}^{R/2} e^{-\alpha \rho} 
\frac{\sinh (\alpha \rho)}{ \cosh (\alpha R)-1} d \varrho \nonumber \\
&\stackrel{\sinh(x) \leq e^x}{\leq}& O(n^2)\cdot  e^{\alpha R (1-\zeta)} \cdot \int_{\delta R}^{R/2}  e^{-\alpha \rho} 
\frac{e^{\alpha \rho}}{ \cosh (\alpha R)-1} d \varrho \nonumber \\
&=& O(n^2) \cdot e^{\alpha R(1-\zeta) - \alpha R} \int_{\delta R}^{R/2} d \varrho  \nonumber \\
&=& O(R) \cdot n^{2} e^{-\alpha  \zeta R} = O(R) \cdot n^{2(1-\alpha \zeta)}\stackrel{\alpha > 1/2}{=} o(n), \nonumber
\end{eqnarray}
provided that $1-\zeta$ is sufficiently small (depending on $\alpha$). 
\end{proof}

We can now consider $\E{X_{y,R/2} (\PPnan)}$. 
Applying the Campbell-Mecke identity~\eqref{eq:Campbell-Mecke} 
on the point process $\PPnan $ on $\D$ with intensity measure $n \cdot \kappa_{\alpha,\nu, n} (\cdot \setminus \mathcal{D}_{\delta R})$,
we have  
\begin{eqnarray} 
\lefteqn{\E{X_{y,R/2} (\PPnan)}=\E{\sum_{p \in \PPnan \cap \A_{y, R/2}} \deg (p; \PPnan)} =} 
\nonumber \\
& & n \cdot  \frac{1}{2\pi}
 \int_{-\pi}^\pi \int_{R/2}^{R-y} \E{\deg ((\varrho, \theta); (\PPnan \cup \{(\varrho, \theta)\}))} \rho_n (\varrho) 
d\varrho d\theta. \label{eq:expectation}
\end{eqnarray}
Now, we bound  the degree of $(\varrho, \theta)$ inside $\A_{y,R/2}$ with respect to the point process $\PPnan \cup \{(\varrho, \theta)\}$ with 
the use of Lemma~\ref{lem:relAngle}.

We apply~\eqref{eq:ball_approx} with $\zeta$ sufficiently close to 1. 
For $\gamma \in (0,1)$, and any finite subset $P\subset \D$ we bound
\begin{eqnarray*} 
|\check{B}_{\zeta,- \gamma} ((\varrho,\theta))\cap P| \leq \deg((\varrho,\theta);P)&\leq& | B^{\uparrow} ((\varrho,\theta);P)|  + |\check{B}_{\zeta, \gamma} ((\varrho,\theta))\cap P|,
\end{eqnarray*}
with 
$$ B^{\uparrow} ((\varrho,\theta);P):= \{p \in P \ : \ y(p) > h_\zeta ((\varrho, \theta)) \}.$$
Let us set
$$ X^{(\zeta,\gamma)}_{y_1,y_2}(P) = \sum_{p\in P \cap \A_{y_1,y_2}} 
 |\check{B}_{\zeta, \gamma} (p)\cap P \setminus\{p\}|.$$
For the expected value of the first term we use the calculation in~\eqref{eq:upper_ball_mass}
which holds uniformly over all $R/2 < \varrho < R-y$: 
\begin{eqnarray}
n \cdot \frac{1}{2\pi} 
\int_{R/2}^{R-y} \int_{-\pi}^\pi \E{ B^{\uparrow} ((\varrho,\theta);(\PPnan \cup \{(\varrho,\theta)\}) \cap \mathcal{D}_{R-y}) } \rho_n (\varrho) 
d\theta d\varrho =o(n) 
\label{eq:degree_up}
\end{eqnarray}
as in the proof of the previous claim,
provided that $1-\zeta$ is sufficiently small (depending on $\alpha$). 

Therefore, 
\begin{equation} \label{eq:approx_balls}
0\leq  \E{X_{r,R/2} (\PPnan)}- \E{X^{(\zeta, \gamma)}_{y,R/2} (\PPnan)} = o(n).
\end{equation}

Now, for any $\gamma \in (-1,1)$, we have 
\begin{eqnarray*}\label{eq:bCheck}
\lefteqn{\E{|\check{B}_{\zeta, \gamma} ((\varrho,\theta))\cap \PPnan| } =} \\
& & n \cdot \frac{\alpha}{2\pi} \cdot  (1+\gamma ) 
2 e^{-R/2 }\cdot e^{\frac12 (R-\varrho)} \int_{2R - \zeta R - \varrho}^{R} e^{\frac12 (R-z)}\frac{\sinh(\alpha z)}{\cosh(\alpha R) -1} d z.
\end{eqnarray*}
Note that since $\varrho > R/2$, for $n$ sufficiently large we have 
$$\left|\frac{\rho_n (z)}{e^{-\alpha (R-z)}}-1\right| =\left|\frac{1}{e^{-\alpha (R-z)}}\cdot \frac{\alpha \sinh (\alpha z)}{\cosh(\alpha R)-1}- 1 \right| < |\gamma|.$$
For real quantities $a, b, c, d$, with $c, d > 0,$ we write that $a = d(b \pm c)$ if $d(b-c) \leq a \leq d(b + c).$ So by the above inequality, the last integral is bounded, for $n$ sufficiently large, as 
\begin{eqnarray} \label{eq:sandwichBound}
\int_{2R - \zeta R - \varrho}^{R} e^{\frac12 (R-z)}\frac{\sinh (\alpha z)}{\cosh(\alpha R) -1} d z =
\frac{(1 \pm |\gamma|)}{\alpha} \cdot
\int_{2R - \zeta R - \varrho}^{R} e^{(\frac12 - \alpha) (R-z) } d z. \end{eqnarray}

By applying the fact that $\varrho > R/2$ and $1-\zeta$ is sufficiently small, hence $\zeta$ is bounded away from $1/2$, we can compute the right hand integral as follows:

\begin{eqnarray*}
\int_{2R - \zeta R - \varrho}^{R} e^{(\frac12 - \alpha) (R-z) } d z
&=&  \int_0^{\zeta R + \varrho - R} e^{(\frac12 - \alpha) z} d z  \\
&\stackrel{\alpha > 1/2}{=}&  \frac{1}{(\alpha - 1/2)} \cdot \left(1 - e^{(\frac{1}{2}-\alpha)(\zeta R + \varrho - R)}\right) \\
&=&
 \frac{1}{(\alpha - 1/2)}  ( 1 - o(1)). \end{eqnarray*}
 
Therefore by substituting this expression into (\ref{eq:sandwichBound}), and taking $n$ to be sufficiently large
for any $\varrho > R/2$
\begin{eqnarray*}
\int_{2R - \zeta R - \varrho}^{R} e^{\frac12 (R-z)}\frac{\sinh (\alpha z)}{\cosh(\alpha R) -1} d z =
\frac{(1 \pm 2|\gamma|)}{\alpha(\alpha- 1/2)} \end{eqnarray*} 

By substituting (\ref{eq:bCheck}) and recalling that $\nu = ne^{-R/2}$, and setting $C_{\alpha, \nu} = \nu / (\pi ( \alpha - 1/2)),$ it follows that uniformly for all $\varrho \in [R/2,R-y]$ and $\theta \in (-\pi,\pi]$ we have:
\begin{eqnarray*} 
\frac{\E{|\check{B}_{\zeta, \gamma} ((\varrho,\theta))\cap \PPnan| }}{e^{\frac12 (R-\varrho)}} 
&=& (1 \pm 2|\gamma|)^{2} C_{\alpha, \nu}, 
\end{eqnarray*}
Therefore, by the Campbell-Mecke formula~\eqref{eq:Campbell-Mecke} we get:
\begin{eqnarray}\label{eq:zetaGammaSandwich}
\E{X^{(\zeta, \gamma)}_{y,R/2} (\PPnan)}&=& \frac{n}{2\pi} 
\int_{R/2}^{R-y} \int_{-\pi}^\pi \E{|\check{B}_{\zeta, \gamma} ((\varrho,\theta))\cap \PPnan| } \rho_n (\varrho) 
d \theta d\varrho,  \nonumber \\
&=&  (1 \pm 2|\gamma|)^{2}  C_{\alpha, \nu} \cdot \frac{n}{2\pi} 
\int_{R/2}^{R-y} \int_{-\pi}^\pi e^{\frac12 (R -\varrho) - \alpha (R-\rho)} d \theta d\varrho.  \nonumber \end{eqnarray} 

Again, we turn our attention to the right hand integral, as $\alpha > 1/2$ and $y < R/4$ we have the following:

\begin{eqnarray}
\int_{R/2}^{R-y} \int_{-\pi}^\pi e^{\frac12 (R -\varrho) - \alpha (R-\rho)} d \theta d\varrho &=& 2\pi \int_{R/2}^{R-y} e^{(1/2 - \alpha) (R-\varrho)} d \varrho, \nonumber \\
&=& 
2 \pi  \int_{y}^{R/2} e^{(1/2 - \alpha)z} d z,  \nonumber 
 \\ 
 &=&   \frac{2 \pi}{\alpha - 1/2}e^{-(\alpha - 1/2)y }\left(1 - e^{-( \alpha - 1/2)(R/2 - y)}\right), \nonumber \\
&=& \frac{2 \pi}{\alpha - 1/2}e^{-(\alpha - 1/2)y }(1 - o(1)) \nonumber,
\end{eqnarray} 
uniformly over $y<R/4$.

Substituting the above into (\ref{eq:zetaGammaSandwich}), and taking $n$ sufficiently large and setting $C'_{\alpha, \nu} = C_{\alpha, \nu}/ (\alpha - 1/2),$ we have the following:

\begin{eqnarray*}
\E{X^{(\zeta, \gamma)}_{y,R/2} (\PPnan)}
&=& n(1 \pm 3 |\gamma|)^{2} C'_{\alpha, \nu} e^{-(\alpha - 1/2)y }.\end{eqnarray*}

So~\eqref{eq:expectation} and~\eqref{eq:approx_balls} yield, for sufficiently large $n$

\begin{equation*}
\E{X_{y,R/2} (\PPnan)} = n (1 \pm 4|\gamma|)^{2}  \cdot C'_{\alpha, \nu}e^{-(\alpha - 1/2)y }.
\end{equation*}

Combining this with Claim~\ref{clm:1st_approx} we deduce the following result: 
for $\gamma \in (-1, 1)$, and $n$ sufficiently large, we have for all $0 \leq y < R/4$, 
\begin{equation} \label{eq:total_degree_high}
n (1 - 5|\gamma|)^{2} C'_{\alpha, \nu}e^{-(\alpha - 1/2)y }  \leq \E{X_{y,R} (\PPnan)} \leq n (1 + 5|\gamma|)^{2}  C'_{\alpha, \nu}e^{-(\alpha - 1/2)y}.
\end{equation}

By applying (\ref{eq:total_degree_high}) we bound the following ratio: 
for $|\gamma|$ chosen small enough such that $(1 + 5|\gamma|)/(1 - 5|\gamma|) < \sqrt{2} $ and $n$ sufficiently large: for all $0 \leq y< R/4$,

\begin{eqnarray*}
\frac{X_{y,R} (\PPnan)}{X_{0,R} (\PPnan)} &\leq& \frac{n (1 + 5|\gamma|)^{2}  C'_{\alpha, \nu}e^{-(\alpha - 1/2)y}}{n (1 - 5|\gamma|)^{2}  C'_{\alpha, \nu}}, \\ &=& \frac{(1 + 5|\gamma|)^{2}}{(1 - 5|\gamma|)^{2}}e^{-(\alpha - 1/2)y}, \\ &\leq& 2e^{-(\alpha - 1/2)y}.
\end{eqnarray*}

\end{proof}

\begin{proof}[Proof of Lemma~\ref{lem:X_conc}]
Since 
$$0\leq \E{X_{y,R}(\PPnan)} -\E{X_{y,R/2}(\PPnan)}=o(n),$$
but $\E{X_{y,R}(\PPnan)} = \Theta (n)$ (for fixed $y>0$) to show the concentration of 
$X_{y,R}(\PPnan)$ around its expected value, it suffices to show that 
$$
 \label{eq:compl_concentration} 
\frac{X_{y,R/2} (\PPnan)}{\E{X_{y,R/2}(\PPnan)}} \to 1 
$$
as $n\to \infty$, in probability. 

We decompose this random variable as follows: 
$$ X_{y,R/2} (\PPnan) = X_{y,\log R} (\PPnan)+ X_{\log R,R/2} (\PPnan).$$
By applying the upper bound of (\ref{eq:total_degree_high}) with $y =\log R$, we deduce that 
$$\E{X_{\log R,R/2} (\PPnan)} = o(n). $$

For $p\in \D$, we set 
$$ \deg_{>h_\zeta(p)} (p;P):=| \{p' \in P \cap B(p;R) \ : \ y(p') > h_\zeta ((\varrho, \theta)) \}|$$
and 
$$ \deg_{\leq h_\zeta (p)} (p;P):= |\{p' \in P \cap B(p;R) \ : \ y(p') \leq  h_\zeta (\varrho, \theta) 
\}|.$$

$$ X_{y,\log R} (\PPnan) = \sum_{p \in \PPnan \cap \A_{y,\log R}} \deg_{>h_\zeta (p)} (p; \PPnan)
+\sum_{p \in \PPnan \cap \A_{y,\log R}} \deg_{\leq h_\zeta (p)} (p;\PPnan). $$ 
Note that $\deg_{>h_\zeta (p)} (p;P) \leq | B^{\uparrow} (p;P)|$. disc
So, by~\eqref{eq:degree_up}, the first term has 
$$ \E{\sum_{p \in \PPnan \cap \A_{y,\log R} } \deg_{>h_\zeta (p)} (p;\PPnan)} =o(n).$$ 
For any $\zeta \in (0,1)$ and any finite set $P \subset \D$, set 
$$X_{y,\log R}^{(\zeta)}(P):=\sum_{p \in P \cap \A_{y, \log R}} \deg_{\leq h_\zeta (p)} (p;P). $$
Therefore, 
$$\E{X_{y,\log R} (\PPnan)} = \E{X_{r,\log R}^{(\zeta)}(\PPnan)} + o(n).$$ 
In turn, 
$$\E{X_{y,\log R}^{(\zeta)}(\PPnan)}=\Theta (n).$$ 
too. 
Hence, to show its concentration around its expected value, it suffices to show that 
$X_{y,\log R}^{(\zeta)}(\PPnan)$ is concentrated around its expected value: as $n\to \infty$
\begin{equation} \label{eq:compl_concentration_I} 
\frac{X_{y,\log R}^{(\zeta)}(\PPnan)}{\E{X_{y,\log R}^{(\zeta)}(\PPnan)}} \to 1, 
\end{equation}
in probability. 
Since the expected value scales linearly in $n$,~\eqref{eq:compl_concentration_I} will follow if we show  that
$$ \Var \left( X_{y,\log R}^{(\zeta)}( \PPnan )\right) = o(n^2).$$

\subsection{Second-moment calculations} 
To bound the variance of $X_{y,\log R}^{(\zeta)}( \PPnan)$, 
we will use Claim~\ref{clm:disjoint_balls}: we set $t_{\zeta, \gamma, R} := 4(1+\gamma) e^{-(1-\zeta) R/2}$ and write ${\A^{2}_{y, \log R}}$ for the product  $\A_{y,\log R} \times \A_{y,\log R}.$

We apply the Campbell-Mecke formula~\eqref{eq:Campbell-Mecke} 
\begin{eqnarray*}
\lefteqn{\frac{(2\pi)^2}{n^2}\E{\left( \sum_{p \in \PPnan \cap  \A_{y,\log R}} \deg_{< h_{\zeta} (p)} (p;\PPnan)\right)^2 } =}\\
& &
\int_{  \A_{y,\log R}^2 }
\E{ \deg_{< h_{\zeta} ((\varrho,\theta))} ((\varrho, \theta)) \cdot \deg_{< h_{\zeta}((\varrho', \theta'))} 
((\varrho', \theta'))
 ; \PPnan \cup \{(\varrho, \theta) , (\varrho', \theta') \}} \times  \\
& &\hspace{2cm} \rho_n (\varrho) \rho_n (\varrho') d \theta' d \varrho' d \theta d \varrho   \\
&=& \int_{\A^{2}_{y,\log R}}
\E{ \deg_{< h_{\zeta}((\varrho, \theta))} ((\varrho,\theta))  \cdot 
\deg_{< h_{\zeta}((\varrho', \theta'))} ((\varrho', \theta')) \cdot 
\mathbf{1}_{|\theta - \theta'|_\pi \leq t_{\zeta, \gamma,R}}; \PPnan \cup \{(\varrho, \theta) , (\varrho', \theta') \}}
\times \\
& & \hspace{2cm} \rho_n (\varrho) \rho_n (\varrho') d \theta' d \varrho' d \theta d \varrho    \\
& &+ \int_{\A^{2}_{y,\log R}} 
\E{ \deg_{< h_{\zeta}((\varrho, \theta))} ((\varrho,\theta))  \cdot \deg_{< h_{\zeta}((\varrho', \theta'))} 
((\varrho', \theta')) \cdot 
\mathbf{1}_{|\theta - \theta'|_\pi > t_{\zeta, \gamma, R} }; \PPnan \cup \{(\varrho, \theta) , (\varrho', \theta') \}} \times \\ 
& &\hspace{2cm} \rho_n (\varrho) \rho_n (\varrho') d \theta' d \varrho' d \theta d \varrho .
\end{eqnarray*}
Recall that for $r>0$, we defined $\A_{r} = \A_{0,R-r}$.
To bound the second integral, let us observe that by Claim~\ref{clm:disjoint_balls}, if 
$|\theta - \theta'|_\pi > t_{\zeta, \gamma, R}$, then 
$$\left( B_R((\varrho,\theta)) \cap \A_{R-h_{\zeta}((\varrho, \theta))} \right) \cap \left( B_R((\varrho', \theta')) \cap \A_{R-h_{\zeta}((\varrho', \theta'))} \right) = \varnothing.$$
So, the random variables 
$ \deg_{< h_{\zeta}((\varrho, \theta))} ((\varrho,\theta);\PPnan\cup \{(\varrho, \theta) , (\varrho', \theta') \})$ and $\deg_{< h_{\zeta}((\varrho', \theta'))} ((\varrho',\theta');\PPnan \cup \{(\varrho, \theta) , (\varrho', \theta') \})$ are independent. Thus, we can write 
\begin{eqnarray*} 
& &\int_{\A^{2}_{y,\log R}} 
\E{ \deg_{< h_{\zeta}((\varrho, \theta))} ((\varrho,\theta))  \cdot \deg_{< h_{\zeta}((\varrho', \theta'))} 
((\varrho', \theta')) \cdot 
\mathbf{1}_{|\theta - \theta'|_\pi > t_{\zeta, \gamma, R} }; \PPnan \cup \{(\varrho, \theta) , (\varrho', \theta') \}} 
\times \\ 
& &\hspace{2cm} \rho_n (\varrho) \rho_n (\varrho') d \theta' d \varrho' d \theta d \varrho  \\
&=&
\int_{\A^{2}_{y,\log R}}
\E{ \deg_{< h_{\zeta}((\varrho, \theta))} ((\varrho,\theta));\PPnan  \cup \{(\varrho, \theta) \} } \times \\
& &\hspace{1cm} \E{\deg_{< h_{\zeta}((\varrho', \theta'))} 
((\varrho', \theta')); \PPnan \cup \{(\varrho', \theta') \}}
\mathbf{1}_{|\theta - \theta'|_\pi > t_{\zeta, \gamma, R} } \cdot  
\rho_n (\varrho) \rho_n (\varrho') d \theta' d \varrho' d \theta d \varrho \\
&\leq& \int_{\A^2_{y,\log R}} 
\E{ \deg_{< h_{\zeta}((\varrho, \theta))} ((\varrho,\theta));\PPnan  \cup \{(\varrho, \theta) \}} \times \\
& &\hspace{2cm} \cdot \E{\deg_{< h_{\zeta}((\varrho', \theta'))} 
((\varrho', \theta')); \PPnan  \cup \{(\varrho', \theta') \}}\rho_n (\varrho) \rho_n (\varrho') d \theta' d \varrho' d \theta 
d \varrho \\
&=&\left( \int_{\A_{y,\log R}}
\E{ \deg_{< h_{\zeta}((\varrho, \theta))} ((\varrho,\theta));\PPnan \cup \{(\varrho, \theta) \} } 
\cdot \rho_n (\varrho) d \varrho d \theta \right)^2  \\
&=& \frac{(2\pi)^2}{n^2}\E{\sum_{p \in \PPnan \cap \A_{y,\log R}} \deg_{< h_{\zeta} (p)} (p)}^2,
\end{eqnarray*}
by the Campbell-Mecke formula~\eqref{eq:Campbell-Mecke}. 

For the first integral, we bound the product of the degrees by the sum of their squares: 
\begin{equation} 
 \deg_{< h_{\zeta}((\varrho, \theta))} ((\varrho,\theta))  \cdot \deg_{< h_{\zeta}((\varrho', \theta'))} 
((\varrho', \theta')) \leq 
 \deg^2_{< h_{\zeta}((\varrho, \theta))} ((\varrho,\theta))  + \deg^2_{< h_{\zeta}((\varrho', \theta'))} 
((\varrho', \theta')).
\end{equation}
So, by symmetry, we bound the first integral as follows: 
\begin{eqnarray} \lefteqn{ \int_{ \A^{2}_{y,\log R}} 
\E{ \deg_{< h_{\zeta}((\varrho, \theta))} ((\varrho,\theta))  \cdot 
\deg_{< h_{\zeta}((\varrho', \theta'))} ((\varrho', \theta')) \cdot 
\mathbf{1}_{|\theta - \theta'|_\pi \leq t_{\zeta, \gamma,R}}; \PPnan \cup \{(\varrho, \theta) , (\varrho', \theta') \}}\times} 
\nonumber \\
& & \hspace{2cm} \rho_n (\varrho) \rho_n (\varrho') d \theta' d \varrho' d \theta d \varrho   \nonumber  \\ 
&\leq&  
2 \cdot  \int_{ \A^{2}_{y,\log R}} 
\E{ \deg^2_{< h_{\zeta}((\varrho, \theta))} ((\varrho,\theta);\PPnan \cup \{(\varrho, \theta) \}) } \cdot
\mathbf{1}_{|\theta - \theta'|_\pi \leq t_{\zeta, \gamma,R}} \times \nonumber \\
& & \hspace{2cm} \rho_n (\varrho) \rho_n (\varrho') d \theta' d \varrho' d \theta d \varrho   \nonumber \\
&=&4 t_{\zeta, \gamma, R} \left( \int_{\A_{y,\log R}} 
\E{ \deg^2_{< h_{\zeta}((\varrho, \theta))} ((\varrho,\theta);\PPnan \cup \{(\varrho, \theta) \})} 
\rho_n (\varrho) d \theta d \varrho \right)\times \nonumber  \\
&&\left( \int_{  \A^{2}_{y,\log R}}   \rho_n (\varrho') d \theta' d \varrho'  \right).
\label{eq:2nd_moment_term_1}
\end{eqnarray}
But by~\eqref{eq:ball_approx}, we have 
$$ \deg_{< h_{\zeta}((\varrho, \theta))} ((\varrho,\theta); \PPnan \cup \{(\varrho, \theta)\}) \leq |\PPnan \cap \check{B}_{\zeta, \gamma} (p)|.$$
So by~\eqref{eq:check_ball_vol_sq} we have 
$$ \E{ \deg^2_{< h_{\zeta}((\varrho, \theta))} ((\varrho,\theta));\PPnan \cup \{(\varrho, \theta) \} } = O(e^{R-\varrho}),
$$
uniformly over all $R-\log R < \varrho < R- y$. 
Therefore, 
\begin{eqnarray}
\lefteqn{ 
\int_{ \A_{y,\log R}} 
\E{ \deg^2_{< h_{\zeta}((\varrho, \theta))} ((\varrho,\theta));\PPnan \cup \{(\varrho, \theta) \} }\rho_n (\varrho) d \theta d \varrho  =} \nonumber \\
& & O(1) \cdot \int_{R-\log R}^{R-y}  e^{R-\varrho} 
\frac{\sinh(\alpha \varrho)}{\cosh(\alpha R)-1} d \varrho  \nonumber \\
&=& O(1) \cdot  \int_{R-\log R}^{R-y}  e^{(R-\varrho) (1-\alpha)}  d \varrho  \nonumber \\
&=& O(1) \cdot \int_y^{\log R} e^{(1-\alpha) z} d z \stackrel{\alpha > 1/2}{=} O(1) \cdot R^{1/2}. \label{eq:square_exp}
\end{eqnarray}
Furthermore, 
\begin{eqnarray}
 \int_{  \A_{y,\log R} }  \rho_n (\varrho') d \theta' d \varrho'  =
 2\pi \frac{\cosh(\alpha (R-y)) - \cosh (\alpha (R-\log R))}{\cosh(\alpha R)-1} = O(1).  \label{eq:leftover}
\end{eqnarray}
Using~\eqref{eq:square_exp} and~\eqref{eq:leftover} into~\eqref{eq:2nd_moment_term_1}, we get 
\begin{eqnarray*} 
&&\int_{ \A^{2}_{y,\log R}} 
\E{ \deg_{< h_{\zeta}((\varrho, \theta))} ((\varrho,\theta))  \cdot 
\deg_{< h_{\zeta}((\varrho', \theta'))} ((\varrho', \theta')) \cdot 
\mathbf{1}_{|\theta - \theta'|_\pi \leq t_{\zeta, \gamma,R}}; \PPnan \cup \{(\varrho, \theta) , (\varrho', \theta') \}}\times 
\nonumber  \\
& & \hspace{2cm} \rho_n (\varrho) \rho_n (\varrho') d \theta' d \varrho' d \theta d \varrho   \nonumber  \\ 
&=& O(1) \cdot t_{\zeta,\gamma,R} R^{1/2} = O(1) \cdot  e^{-(1-\zeta) R/2} R^{1/2} \nonumber \\
&=& O(1) \cdot n^{-(1-\zeta)} R^{1/2}. \label{eq:2ndmoment_term1_final}
\end{eqnarray*} 
Therefore, we obtain
\begin{eqnarray*}
\E{\left( \sum_{p \in \PPnan \cap  \A_{y,\log R} } \deg_{< h_{\zeta} (p)} (p)\right)^2 } &\leq& \E{\sum_{p \in \PPnan \cap  \A_{y,\log R} } \deg_{< h_{\zeta} (p)} (p)}^2  \\
&& \hspace{2cm}+ O(1) \cdot n^{2 - (1-\zeta)} R^{1/2}.
\end{eqnarray*}
Rearranging the above, we get
\begin{equation*} \label{eq:variance}
\Var \left( \sum_{p \in \PPnan \cap  \A_{y,\log R}} \deg_{< h_{\zeta} (p)} (p)\right) 
=O(1) \cdot n^{1+\zeta} R^{1/2} = o(n^2). 
\end{equation*}
\end{proof}

\section{Discussion}

In this paper we have considered the modularity score of the KPKVB model of the hyperbolic random graph. We have shown that for all $\alpha > 1/2$ and $\nu > 0$ we have that $\q({\Pnan}) \rightarrow 1$ as $n \rightarrow \infty$ in probability. The partition we consider is that of dividing the Poincar\'{e} disc into a constant number of equal sectors. We show that the modularity of this partition is closely related to the box partition given in $\Bnan{y}$. Following from this, we observe that for any $\varepsilon > 0$ a.a.s the modularity of $\Bnan{y}$ is at least $1 - \varepsilon$ and thus $\q (\Pnan) \to 1$,  as $n\to \infty$, in probability. 

One question raised by the last author and McDiarmid, is the order of $1-\q(G),$ also referred to as the \emph{modularity deficit} \cite{ERgraphs}. The modularity deficit quantifies how much a given partition differs from optimal modularity. While we deduce that the modularity deficit of the sector division can be made arbitrarily small, it is open to determine whether we can explicitly express the rate of convergence asymptotically. It is also to determine for a given growth rate, whether we can exhibit a partition that possess such a deficit.

A modular community structure is characterised by a vertex partition where edge density within parts is much greater than expected, while density between parts is much smaller. While a high a modularity score $( >0.3)$ can be indicative of an underlying modular community structure, a high score alone does not guarantee that such a community structure exists. This tends to occur in sparse
networks. For example, in regimes where the average degree is bounded, the Erd\H{o}s-R\'{e}ny\'{i} random graph can exhibit a high modularity score in probability, without possessing a modular community structure~\cite{ERgraphs}.  

In the case of the KPKBV model, the high modularity may be a consequence of the tree-like structure of the random graph. 
Generally, trees with sublinear maximum degree demonstrate an almost optimal modularity score; see \cite{REGgraphs}.
Here, the term ``tree-like'' does not refer to the lack of short cycles (in fact, the presence of clustering implies that there are many short cycles with high probability). 
It refers to the existence of a hierarchy on the set of vertices of the random graph, 
which resembles the natural hierarchy that a rooted tree exhibits.
Let us note that as a consequence of the negative curvature of hyperbolic space, tangential distances in the Poincar\'e  disc expand exponentially with the respect to the radial distance from the centre. Pairs of vertices near the boundary of the disc are much less likely to connect, as they must possess a much smaller relative angle for this to happen. In contrast, vertices near the centre have relatively high degree, as the balls of radius $R$ around them cover almost all of the disc. This means that the communities tend to have an underlying hierarchical structure, where the communities are formed from the mutual descendants of nodes with larger defect radii. 
Each part of the sector partition tends to capture a large proportion of one of these rooted sub-trees; therefore, this may suggest why the modularity score of the sector partition tends to one, in probability.

\bibliographystyle{plain}

%

%
%
%

\bibliography{mod}

\def\cprime{$'$}
\begin{thebibliography}{10}

\bibitem{ar:AbdBodeFound}
M.A. Abdullah, M.~Bode, and N.~Fountoulakis.
\newblock Typical distances in a geometric model for complex networks.
\newblock {\em Internet Mathematics}, 1, 2017.

\bibitem{BarAlb}
R.~Albert and A.-L. Barab\'asi.
\newblock Statistical mechanics of complex networks.
\newblock {\em Rev. Mod. Phys.}, 74(1):47--97, 2002.

\bibitem{ar:BlondelFast}
V.~Blondel, J.L. Guillaume, R.~Lambiotte, and E.~Lefebvre.
\newblock Fast unfolding of communities in large networks.
\newblock {\em Journal of statistical mechanics: theory and experiment},
  2008(10):P10008, 2008.

\bibitem{BFMgiantEJC}
M.~Bode, N.~Fountoulakis, and T.~M{\"u}ller.
\newblock On the largest component of a hyperbolic model of complex networks.
\newblock {\em Electronic Journal of Combinatorics}, 22(3), 2015.
\newblock Paper P3.24, 43 pages.

\bibitem{BFMconnRSA}
M.~Bode, N.~Fountoulakis, and T.~M{\"u}ller.
\newblock The probability of connectivity in a hyperbolic model of complex
  networks.
\newblock {\em Random Structures Algorithms}, 49(1):65--94, 2016.

\bibitem{ar:Brandes2007}
U.~Brandes, D.~Delling, M.~Gaertler, R.~G{\"o}rke, M.~Hoefer, Z.~Nikoloski, and
  D.~Wagner.
\newblock On finding graph clusterings with maximum modularity.
\newblock In {\em Proceedings of the 33rd International Workshop on
  {G}raph-theoretic {C}oncepts in Computer Science}, volume 4769 of {\em
  Lecture Notes in Computer Science}, pages 121--132, 2007.

\bibitem{dinh2015network}
T.~Dinh, X.~Li, and M.~Thai.
\newblock Network clustering via maximizing modularity: Approximation
  algorithms and theoretical limits.
\newblock In {\em 2015 IEEE International Conference on Data Mining}, pages
  101--110. IEEE, 2015.

\bibitem{ar:FM_AAP}
N.~Fountoulakis and T.~M\"uller.
\newblock Law of large numbers in a hyperbolic model of complex networks.
\newblock {\em Annals of Applied Probability}, 28:607--650, 2018.

\bibitem{ar:CF2021}
N.~Fountoulakis, P.~van~der Hoorn, T.~M\"uller, and M.~Schepers.
\newblock "clustering in a hyperbolic model of complex networks.
\newblock {\em Electronic Journal of Probability}, to appear:pp. 126, 2020.

\bibitem{ar:NFJY2020}
N.~Fountoulakis and J.~Yukich.
\newblock Limit theory for isolated and extreme points in hyperbolic random
  geometric graphs.
\newblock {\em Electronic Journal of Probability}, 25:pp. 51, 2020.

\bibitem{ar:FriedKrohmerDiam}
T.~Friedrich and A.~Krohmer.
\newblock On the diameter of hyperbolic random graphs.
\newblock {\em SIAM J. Disc. Math.}, 32:1314--1334, 2018.

\bibitem{ar:Gran73}
M.~Granovetter.
\newblock The strength of weak ties.
\newblock {\em American Journal of Sociology}, 78:1360, 1973.

\bibitem{ar:Gugel}
L.~Gugelmann, K.~Panagiotou, and U.~Peter.
\newblock Random hyperbolic graphs: Degree sequence and clustering.
\newblock In {\em Proceedings of the 39th International Colloquium Conference
  on Automata, Languages, and Programming - Volume Part II}, ICALP'12, pages
  573--585, Berlin, Heidelberg, 2012. Springer-Verlag.

\bibitem{KiwiMit}
M.~A. Kiwi and D.~Mitsche.
\newblock A bound for the diameter of random hyperbolic graphs.
\newblock In Robert Sedgewick and Mark~Daniel Ward, editors, {\em Proceedings
  of the Twelfth Workshop on Analytic Algorithmics and Combinatorics, {ANALCO}
  2015, San Diego, CA, USA, January 4, 2015}, pages 26--39. {SIAM}, 2015.

\bibitem{KiwiMit2017+}
M.~A. Kiwi and D.~Mitsche.
\newblock On the second largest component of random hyperbolic graphs.
\newblock {\em SIAM J. Discrete Math.}, 33(4):2200--2217, 2019.

\bibitem{ar:Krioukov}
D.~Krioukov, F.~Papadopoulos, M.~Kitsak, A.~Vahdat, and M.~Bogu{\~n}{\'a}.
\newblock Hyperbolic geometry of complex networks.
\newblock {\em Phys. Rev. E (3)}, 82(3):036106, 18, 2010.

\bibitem{ar:lancichinetti}
A.~Lancichinetti and S.~Fortunato.
\newblock Limits of modularity maximization in community detection.
\newblock {\em Physical review E}, 84(6):066122, 2011.

\bibitem{bk:LastPenrose}
G.~Last and M.~Penrose.
\newblock {\em Lectures on the Poisson Process}.
\newblock IMS Textbooks. Cambridge University Press, 2018.

\bibitem{ar:LubMit2020}
L.~Lichev and D.~Mitsche.
\newblock On the modularity of 3-regular graphs and random graphs with given
  degree sequences.
\newblock 41pp, \texttt{arXiv 2007:15574v1}, 2020.

\bibitem{REGgraphs}
C.J.H. McDiarmid and F.~Skerman.
\newblock Modularity of regular and tree-like graphs.
\newblock {\em Journal of Complex Networks}, 4(6):596--619, 2018.

\bibitem{ERgraphs}
C.J.H. McDiarmid and F.~Skerman.
\newblock Modularity of {E}rd{\H{o}}s-{R}\'enyi random graphs.
\newblock {\em Random Structures and Algorithms}, 57(1):211--243, 2020.

\bibitem{ar:MullerDiam}
T.~M\"uller and M.~Staps.
\newblock The diameter of {KPKVB} random graphs.
\newblock {\em Advances in Applied Probability}, 51(2):358--377, 2019.

\bibitem{ar:Newm2006}
M.E.J. Newman.
\newblock Finding community structure in networks using the eigenvectors of
  matrices.
\newblock {\em Phys. Rev. E}, 74:036104, 2006.

\bibitem{arNewm2006PNAS}
M.E.J. Newman.
\newblock Modularity and community structure in networks.
\newblock {\em Proceedings of the National Academy of Sciences},
  103(23):8577--8582, 2006.

\bibitem{NewmanGirvan}
M.E.J. Newman and M.~Girvan.
\newblock Finding and evaluating community structure in networks.
\newblock {\em Phys. Rev. E}, 69:026113, 2004.

\bibitem{bk:Penrose}
M.~D. Penrose.
\newblock {\em Random geometric graphs}, volume~5 of {\em Oxford Studies in
  Probability}.
\newblock Oxford University Press, Oxford, 2003.

\bibitem{ar:CF_vdH2020}
C.~Stegehuis, R.~v.d. Hofstad, and J.S.H.~v. Leeuwaarden.
\newblock Scale-free network clustering in hyperbolic and other random graphs.
\newblock {\em {J}ournal of {P}hysics {A}: {M}athematical and {T}heoretical},
  52(29):295101, 2019.

\end{thebibliography}
\end{document}